\documentclass[a4paper,10pt]{amsart}
\usepackage[english]{babel}
\usepackage[leqno]{amsmath}
\usepackage{amssymb,amsthm}
\usepackage{amscd}
\usepackage{enumerate}
\usepackage{epsfig}

\input{xy}
\xyoption{all}

\newtheorem{thm}{Theorem}[subsection]
\newtheorem{lemma}[thm]{Lemma}
\newtheorem{lemmadefi}[thm]{Lemma - Definition}

\newtheorem{prop}[thm]{Proposition}

\newtheorem{cor}[thm]{Corollary}

\newtheorem{fact}[thm]{Fact}

\theoremstyle{remark}
\newtheorem{remark}[thm]{Remark}
\theoremstyle{definition}
\newtheorem{defi}[thm]{Definition}
\newtheorem{nota}[thm]{}
\newtheorem{constr}[thm]{Construction}

\numberwithin{equation}{section}

\newenvironment{sis}{\left\{\begin{aligned}}{\end{aligned}\right.}

\newtheorem{example}[thm]{Example}

\newcommand{\w}{\widetilde}
\newcommand{\ov}{\overline}
\newcommand{\un}{\underline}

\renewcommand{\S}{\mathcal{S}}

\newcommand{\Ker}{\operatorname{Ker}}
\renewcommand{\Im}{\operatorname{Im}}

\newcommand{\val}{\operatorname{val}}
\newcommand{\Stab}{\operatorname{Stab}}

\newcommand{\Alb}{\operatorname{Alb}}

\newcommand{\Jac}{\operatorname{Jac}}
\newcommand{\Aut}{\operatorname{Aut}}

\newcommand{\GL}{\operatorname{GL}}

\newcommand{\Q}{\mathbb{Q}}
\newcommand{\Z}{\mathbb{Z}}
\newcommand{\R}{\mathbb{R}}
\newcommand{\N}{\mathbb{N}}

\newcommand{\CC}{\mathcal{C}}
\newcommand{\DD}{\mathcal{D}}
\newcommand{\BB}{\mathcal{B}}
\newcommand{\II}{\mathcal{I}}
\renewcommand{\L}{\mathcal{L}}
\newcommand{\A}{\mathcal{A}}

\newcommand{\girth}{{\rm girth}}

\newcommand{\Vor}{{\rm Vor}}

\newcommand{\prin}{\sigma_{{\rm prin}}}
\newcommand{\Cpr}{C_{\rm prin}}

\renewcommand{\O}{\Omega_g}
\newcommand{\Ort}{\Omega_g^{\rm rt}}

\newcommand{\Null}{{\rm Null}}

\newcommand{\Mgt}{M_g^{\rm tr}}

\newcommand{\Mgpl}{M_g^{{\rm tr, pl}}}

\newcommand{\Agt}{A_g^{\rm tr}}
\newcommand{\Agcom}{A_g^{{\rm tr}, \Sigma}}
\newcommand{\AgV}{A_g^{{\rm tr, V}}}
\newcommand{\Agzon}{A_g^{{\rm zon}}}
\newcommand{\Aggr}{A_g^{\rm  gr}}
\newcommand{\Agco}{A_g^{\rm cogr}}
\newcommand{\Agcogr}{A_g^{\rm gr, cogr}}

\newcommand{\tgt}{t_g^{\rm tr}}

\newcommand{\Mg}{\mathcal{M}_g}
\newcommand{\Mgb}{\ov{\mathcal{M}_g}}
\newcommand{\Ag}{\mathcal{A}_g}
\newcommand{\Agb}{\ov{\mathcal{A}_g}^V}
\newcommand{\tg}{{\rm t}_g}
\newcommand{\tgb}{{\ov{\rm t}_g}}
\newcommand{\Mgpla}{\mathcal{M}_g^{\rm pl}}
\newcommand{\Agcogra}{\mathcal{A}_g^{\rm cogr}}
\newcommand{\Aggrco}{\mathcal{A}_g^{\rm gr,cogr}}
\newcommand{\Sg}{\mathcal{S}_g}
\newcommand{\Rg}{\mathcal{R}_g}
\newcommand{\Ng}{\mathcal{N}_g}

\newcommand{\Int}{\operatorname{Int}}

\newcommand{\Zon}{\operatorname{Zon}}
\newcommand{\Cogr}{\operatorname{Cogr}}
\newcommand{\Gr}{\operatorname{Gr}}
\newcommand{\Grcogr}{\operatorname{Gr-cogr}}

\begin{document}

\title{On the tropical Torelli map}

\author{Silvia Brannetti}
\address{Dipartimento di Matematica,
Universit\`a Roma Tre,
Largo S. Leonardo Murialdo 1,
00146 Roma (Italy)}
\email{brannett@mat.uniroma3.it}

\author{Margarida Melo}
\address{Departamento de Matem\'atica da Universidade de Coimbra,
Largo D. Dinis, Apartado 3008, 3001 Coimbra (Portugal)}
\email{mmelo@mat.uc.pt}

\author{Filippo Viviani}
\address{
Dipartimento di Matematica,
Universit\`a Roma Tre,
Largo S. Leonardo Murialdo 1,
00146 Roma (Italy)}
\email{viviani@mat.uniroma3.it}

\thanks{The second author was partially supported by a Funda\c{c}\~ao para a Ci\^encia e Tecnologia doctoral grant. The third author was partially supported by FCT-Ci\^encia2008.}



\begin{abstract}
We construct the moduli spaces of tropical curves and tropical
principally polarized abelian varieties, working in the
category of (what we call) stacky fans.
We define the tropical Torelli map between these two moduli spaces
and we study the fibers (tropical Torelli theorem) and the image of this map
(tropical Schottky problem).
Finally we determine the image of the planar tropical curves
via the tropical Torelli map and we use it to give a positive
answer to a question raised by Namikawa on the compactified classical
Torelli map.
\end{abstract}

\maketitle



\section{Introduction}

\begin{nota}{\emph{The problem}}

The classical Torelli map $\tg:\Mg\to \Ag$
is the modular map from the
moduli space $\Mg$ of smooth and projective curves of genus $g$
to the moduli space $\Ag$ of principally polarized abelian varieties
of dimension $g$, sending a curve $C$ into its Jacobian variety
$\Jac(C)$, naturally endowed with the principal polarization given
by the class of the theta divisor $\Theta_C$. The Torelli map
has been widely studied
as it allows to relate the study of curves
to the study of linear (although higher-dimensional)
objects, i.e. abelian varieties.
Among the many known results on the Torelli map $\tg$, we
mention: the injectivity of the map $\tg$
(proved by Torelli in \cite{torelli}) and the many different
solutions to the so-called Schottky problem, i.e. the
problem of characterizing the image of $\tg$ (see the nice survey of
Arbarello in the appendix of \cite{Mum}).

The aim of this paper is to define and study a tropical analogous
of the Torelli map. Tropical geometry is a recent branch of
mathematics that establishes deep relations
between algebro-geometric and purely combinatorial objects.
For an introduction to tropical geometry,
see the surveys \cite{MIK4}, \cite{SS1},
\cite{Stur}, \cite{Gat}, \cite{MIK3}, \cite{Kat},
\cite{MIK0}, or the books in preparation \cite{MIK},
\cite{McLS}.

Ideally, every construction in algebraic geometry should have a combinatorial
counterpart in tropical geometry. One may thus hope to obtain results in algebraic geometry
by looking at the tropical (i.e. combinatorial) picture first and then trying to transfer the results back to
the original algebro-geometric setting. For instance,
this program has been carried
out successfully for many problems of real and
complex enumerative geometry, see for example \cite{IKS},
\cite{MIK5}, \cite{NS}, \cite{GM2}, \cite{GM}, \cite{CJM}, \cite{MR}, \cite{BM}, \cite{FM}.


In the paper \cite{MZ}, Mikhalkin and Zharkov
studied abstract tropical curves and tropical abelian varieties.
They construct the Jacobian $\Jac(C)$ and observe that the naive generalization of the
Torelli theorem, namely that a curve $C$ is determined by its
Jacobian $\Jac(C)$, is false in this tropical setting.
However, they speculate that this naive generalization
should be replaced by the statement that the tropical Torelli
map $\tgt:\Mgt\to \Agt$
has tropical degree one, once it has been properly defined!

In \cite{CV1}, Caporaso and Viviani determine when two tropical
curves have the same Jacobians. They use this to prove that the tropical
Torelli map is indeed of tropical degree one, assuming the existence of
the moduli spaces $\Mgt$ and $\Agt$
as well as the existence of the tropical Torelli map $\tgt:\Mgt\to
\Agt$, subject to some natural properties.
Indeed, a construction of the
moduli spaces $\Mgt$ and $\Agt$ for every $g$ remained open
so far, at least to our knowledge. Though, the moduli
space of $n$-pointed tropical rational curves $M_{0, n}^{\rm tr}$
was constructed by different authors
(see \cite{SS2}, \cite{MIK1}, \cite{GKM}, \cite{KM}).
\end{nota}

\begin{nota}{\emph{The results}}

The aim of the present paper is to define the moduli spaces
$\Mgt$ and $\Agt$, the tropical Torelli map
$\tg:\Mgt\to \Agt$ and to investigate an analogue of the
Torelli theorem and of the Schottky problem.

With that in mind, we introduce slight generalizations in the definition
of tropical curves and tropical principally polarized abelian varieties.
Throughout this paper, a \emph{tropical curve} $C$ of genus $g$ is given by a marked
metric graph $(\Gamma,w,l)$,
where $(\Gamma,l)$ is a metric graph and $w:V(\Gamma)\to \Z_{\geq 0}$ is
a weight function defined on the set $V(\Gamma)$ of vertices of
$\Gamma$, such that $g=b_1(\Gamma)+
|w|$, where $|w|:=\sum_{v\in V(\Gamma)} w(v)$ is the total weight of the
graph, and the marked graph $(\Gamma,w)$ satisfies a
stability condition (see Definitions \ref{mark-graph} and
\ref{pseudo-trop}).
A (principally polarized) \emph{tropical abelian variety}
$A$ of dimension $g$ is a real torus $\R^g/\Lambda$ as before,
together with a flat semi-metric coming from a positive
semi-definite quadratic form $Q$ with rational null-space
(see Definition \ref{ps-abvar}).
To every tropical curve $C=(\Gamma,w,l)$ of genus $g$, it is
associated a tropical abelian variety of dimension $g$, called
the \emph{Jacobian} of $C$ and denoted by $\Jac(C)$, which is given
by the real torus $(H_1(\Gamma,\R)\oplus \R^{|w|})/(H_1(\Gamma,\Z)
\oplus \Z^{|w|})$, together with the positive semi-definite
quadratic form $Q_{(\Gamma,l)}$ which vanishes on $\R^{|w|}$ and
is given on $H_1(\Gamma,\R)$ by $Q_{(\Gamma,l)}(\sum_{e\in
E(\Gamma)}n_e \cdot e)=\sum_{e\in E(\Gamma)} n_e^2 \cdot l(e)$.
The advantage of such a generalization in the definition of tropical curves and
tropical abelian varieties is that the moduli spaces we will construct
are closed under specializations (see subsection \ref{ps-trop} for more details).



The construction of the moduli spaces of tropical
curves and tropical abelian varieties is performed within
the category of what we call \emph{stacky fans} (see section \ref{manifold}).
A stacky fan is, roughly speaking,
a topological space given by a collection of quotients of rational polyhedral cones, called
cells of the stacky fan,
whose closures are glued together along their boundaries via integral
linear maps (see definition \ref{pol-compl}).

The moduli space $\Mgt$ of tropical curves of genus $g$ is a stacky fan
with cells $C(\Gamma,w)=\R_{> 0}^{|E(\Gamma|)}/\Aut(\Gamma,w)$,
where $(\Gamma,w)$ varies among stable marked graphs of genus $g$,
consisting of all the tropical curves whose underlying
marked graph is equal to $(\Gamma,w)$ (see definition \ref{Mgt}).
The closures of two cells
$\ov{C(\Gamma,w)}$ and $\ov{C(\Gamma',w')}$ are glued together along the faces
that correspond
to common specializations of $(\Gamma,w)$ and $(\Gamma',w')$
(see Theorem \ref{Mgt-pol}).
Therefore, in $\Mgt$, the closure of a cell $C(\Gamma,w)$ will be equal to a
disjoint union of lower dimensional cells $C(\Gamma',w')$ corresponding to
different specializations of $(\Gamma,w)$.

We describe the maximal cells and the codimension one
cells of $\Mgt$ and we prove that $\Mgt$ is pure dimensional and connected through codimension
one (see Proposition \ref{prop-Mgt}). Moreover the topology with which $\Mgt$ is endowed is shown in \cite{Cap} to be Hausdorff. A Teichm\"{u}ller-type approach to the construction of $\Mgt$ was suggested by Mikhalkin in \cite[3.1.]{MIK0}, using the theory of Outer Spaces from Geometric Group Theory. It would be very interesting to investigate this different approach and compare it to ours.

The moduli space $\Agt$ of tropical abelian varieties of
dimension $g$ is first constructed as a topological space
by forming the quotient $\Ort/\GL_g(\Z)$, where
$\Ort$ is the cone of positive semi-definite quadratic forms
in $\R^g$ with rational null space and the action of
$\GL_g(\Z)$ is via the usual arithmetic equivalence (see definition
\ref{Agtr}).
In order to put a structure of stacky fan on $\Agt$, one has to specify a $\GL_g(\Z)$-admissible
decomposition $\Sigma$ of $\Ort$ (see definition \ref{decompo}),
i.e. a fan decomposition of $\Ort$ into (infinitely many) rational
polyhedral cones that are stable under the action of $\GL_g(\Z)$
and such that there are finitely many equivalence classes of cones
modulo $\GL_g(\Z)$. Given such a $\GL_g(\Z)$-admissible
decomposition $\Sigma$ of $\Ort$, we endow $\Agt$ with the
structure of a stacky fan, denoted by $\Agcom$,
in such a way that the cells of $\Agcom$ are exactly the
$\GL_g(\Z)$-equivalence
classes of cones in $\Sigma$ quotiented out by their
stabilizer subgroups (see Theorem \ref{AgS}).

Among all the known $\GL_g(\Z)$-admissible decompositions of $\Ort$,
one will play a special role in this paper, namely the (second)
Voronoi decomposition which we denote by $V$. The cones of
$V$ are formed by those elements $Q\in \Ort$ that have the same
Dirichlet-Voronoi polytope $\Vor(Q)$
(see definition \ref{Vor-dec}). We denote the corresponding
stacky fan by $\AgV$ (see definition
\ref{AgV}). We describe the maximal cells and the codimension one
cells of $\AgV$ and we prove that $\AgV$ is pure-dimensional and connected
through codimension one (see Proposition \ref{prop-AgV}).
$\AgV$ admits an important stacky subfan,
denoted by $\Agzon$, formed by all the
cells of $\AgV$ whose associated Dirichlet-Voronoi polytope
is a zonotope.
We show that $\GL_g(\Z)$-equivalence classes of zonotopal Dirichlet-Voronoi polytopes
(and hence the cells of $\Agzon$) are in bijection with simple
matroids of rank at most $g$ (see Theorem \ref{mat-zon}).

After having defined $\Mgt$ and $\AgV$, we show that the tropical
Torelli map
$$\begin{aligned}
\tgt: \Mgt & \to \AgV\\
C & \mapsto \Jac(C),
\end{aligned}$$
is a map of stacky fans (see
Theorem \ref{tr-map}).

We then prove a Schottky-type and a Torelli-type theorem
for $\tgt$. The Schottky-type theorem says that
$\tgt$ is a full map whose image is equal to the stacky subfan
$\Agcogr\subset \Agzon$, whose cells correspond
to cographic simple matroids of rank at most $g$ (see Theorem
\ref{Sch}). The Torelli-type theorem says that $\tgt$ is of degree
one onto its image (see Theorem \ref{deg-one}). Moreover,
extending the results of Caporaso and Viviani \cite{CV1} to our
generalized tropical curves (i.e. admitting also weights),
we determine when two tropical
curves have the same Jacobian (see Theorem \ref{tor}).

Finally, we define the stacky subfan $\Mgpl\subset \Mgt$
consisting of planar tropical curves (see definition
\ref{Mg-planar}) and
the stacky subfan $\Aggr\subset \Agzon$ whose cells correspond
to graphic simple matroids of rank at most $g$
(see definition \ref{Gra}).
We show that $\Aggr$ is also equal
to the closure inside $\AgV$ of the so-called principal cone
$\prin^0$ (see Proposition \ref{desc-gr}).
We prove that $\tgt(C)\in \Aggr$ if and only if $C$ is a planar
tropical curve and that
$\tgt(\Mgpl)=\Agcogr:=\Agco \cap \Aggr$ (see Theorem \ref{tor-planar}).

As an application of our tropical results, we study a problem raised by Namikawa
in \cite{NamT} concerning the extension $\tgb$ of the (classical) Torelli map
from the Deligne-Mumford compactification $\Mgb$ of $\Mg$ to the (second)
Voronoi toroidal compactification $\Agb$ of $\Ag$ (see subsection
\ref{Nam-section} for more details).
More precisely, in Corollary \ref{Nam-conj}, we provide
a characterization of the stable curves whose dual graph is planar
in terms of their image via the compactified Torelli map $\tgb$, thus answering affirmatively
to \cite[Problem (9.31)(i)]{NamT}. The relation between our tropical moduli spaces
$\Mgt$ (resp. $\AgV$) and the compactified moduli spaces $\Mgb$ (resp. $\Agb$) is that
there is a natural bijective correspondence between the cells of the former
and  the strata of the latter; moreover these bijections are compatible
with the Torelli maps $\tgt$ and $\tgb$. This allows us to apply our results about
$\tgt$ to the study of $\tgb$, providing thus the necessary tools to solve Namikawa's problem.

\end{nota}

\begin{nota}{\emph{Outline of the paper}}

In section 2, we collect all the preliminaries that we will
need in the sequel. We first define the category of stacky fans.
Then, for the reader's convenience, we review the concepts of
graph theory and (unoriented) matroid theory, that will play a major
role throughout the paper.

In section 3, we define tropical curves and construct
the moduli space $\Mgt$ of tropical curves of genus $g$.

In section 4, we first define tropical (principally polarized)
abelian varieties and then we
construct the moduli space $\Agt$ of tropical abelian
varieties. We show how to endow $\Agt$ with the structure of a stacky fan $\Agcom$ for every
$\GL_g(\Z)$-admissible decomposition $\Sigma$ of
the cone $\Ort$.
Then, we focus our attention on the (second) Voronoi
$\GL_g(\Z)$-admissible decomposition of $\Ort$ and the resulting
stacky fan structure on $\Agt$, which
we denote by $\AgV$. We define a stacky subfan
$\Agzon\subset \AgV$
whose cells correspond to $\GL_g(\Z)$-equivalence classes of
zonotopal Dirichlet-Voronoi polytopes,
and we show that these cells are in bijection with simple
matroids of rank at most $g$.

In section 5, we define the tropical Torelli map
$\tgt:\Mgt \to \AgV$. We prove a Schottky-type theorem and a
Torelli-type theorem.

In section 6, we study the restriction of the tropical Torelli
map $\tgt$ to the stacky subfan $\Mgpl\subset \Mgt$
of planar tropical curves and we give a positive answer
to Namikawa's question.

In section 7, we list some of the possible further developments of our work. We
hope to come back to some of them in a near future.

\end{nota}

\emph{Acknowledgements.}

This project started while the three authors were visiting
the MSRI, in Berkeley, for the special semester on Algebraic Geometry
in 2009: we thank the organizers of the program as well as the Institution
for the stimulating atmosphere.
We thank L. Caporaso, G. Mikhalkin, B. Sturmfels, O. Tommasi, I. Yudin for useful
conversations. We thank the referee for useful remarks.

\section{Preliminaries}

\subsection{Stacky fans}
\label{manifold}

In order to fix notations, recall some concepts from convex geometry.
A \emph{polyhedral cone} $\Xi$ is the
intersection of finitely many closed linear half-spaces in $\R^n$.
The \emph{dimension} of $\Xi$ is the dimension of the smallest linear subspace
containing $\Xi$. Its \emph{relative interior} $\Int\Xi$ is the
interior inside this linear subspace, and the complement $\Xi\setminus
\Int\Xi$ is called the \emph{relative boundary} $\partial\,\Xi$.
If $\dim\Xi=k$ then $\partial\,\Xi$ is itself a union of polyhedral cones of
dimension at most $k-1$, called \emph{faces}, obtained by intersection
of $\Xi$ with linear hyperplanes disjoint from $\Int\Xi$. Faces of dimensions
$k-1$ and $0$
are called \emph{facets} and
\emph{vertices}, respectively.
A polyhedral cone is \emph{rational} if the linear functions defining the
half-spaces can be taken with rational coefficients.

An \emph{open polyhedral cone} of $\R^n$ is the relative
interior of a polyhedral cone.
Note that the closure
of an open polyhedral cone
with respect to the Euclidean topology of $\R^n$ is a polyhedral
cone.
An open polyhedral cone is rational if its closure is rational.

We say that a map $\R^n\to \R^m$ is \emph{integral linear}
if it is linear and sends $\Z^n$ into $\Z^m$, or equivalently
if it is linear and can be represented by an integral matrix
with respect to the canonical bases of $\R^n$ and $\R^m$.


\begin{defi}\label{pol-compl}
Let $\{X_k\subset \R^{m_k}\}_{k\in K}$ be a finite collection of rational open
polyhedral cones such that $\dim X_k=m_k$. Moreover, for each such cone $X_k\subset \R^{m_k}$,
let $G_k$ be a group and $\rho_k:G_k\to \GL_{m_k}(\Z)$ a homomorphism
such that
$\rho_k(G_k)$ stabilizes the cone $X_k$ under its natural
action on $\R^{m_k}$. Therefore $G_k$ acts on $X_k$ (resp. $\ov{X_k}$),
via the homomorphism $\rho_k$, and we denote the quotient by $X_k/G_k$
(resp. $\ov{X_k}/G_k$), endowed with the quotient topology.
A topological space $X$ is said to be a \emph{stacky (abstract) fan}
with cells $\{X_k/G_k\}_{k\in K}$ if there exist continuous maps
$\alpha_k:\ov{X_k}/G_k\to X$
satisfying the following pro\-per\-ties:
\begin{enumerate}[(i)]
\item \label{pol1} The restriction of $\alpha_k$ to $X_k/G_k$
is an homeomorphism onto its image;
\item \label{pol2} $X=\coprod_k \alpha_k(X_k/G_k)$ (set-theoretically);
\item \label{pol3} For any $j, k\in K$, the natural inclusion map
$\alpha_k(\ov{X_k}/G_k)\cap \alpha_j(\ov{X_j}/G_j)\hookrightarrow
\alpha_j(\ov{X_j}/G_j)$ is induced by an integral linear map
$L:\R^{m_k}\to \R^{m_j}$, i.e. there exists a commutative diagram
\begin{equation}
\xymatrix{
\alpha_k(\ov{X_k}/G_k)\cap \alpha_j(\ov{X_j}/G_j) \ar@{_{(}->}[rd] \ar@{^{(}->}[r]&
\alpha_k (\ov{X_k}/G_k)& \ov{X_k} \ar@{^{(}->}[r] \ar@{->>}[l]
\ar[d]^L& \R^{m_k} \ar[d]^{L} \\
&\alpha_j(\ov{X_j}/G_j)& \ov{X_j} \ar@{^{(}->}[r] \ar@{->>}[l]
& \R^{m_j} .}
\end{equation}

\end{enumerate}
By abuse of notation, we usually identify $X_k/G_k$ with its image inside
$X$ so that we usually write $X=\coprod X_k/G_k$ to denote the decomposition
of $X$ with respect to its cells $X_k/G_k$.

A stacky \emph{subfan} of $X$ is a closed subspace $X'\subseteq X$
that is a disjoint union of cells of $X$. Note that $X'$
inherits a natural structure of stacky fan
with respect to the sub-collection
$\{X_{k}/G_k\}_{k\in K'}$
of cells that are contained in $X'$.


The dimension of $X$, denoted by $\dim X$, is the greatest dimension of its cells.
We say that a cell is maximal if it is not contained in the closure
of any other cell. $X$ is said to be of \emph{pure dimension} if all its maximal
cells have dimension equal to $\dim X$. A \emph{generic point} of $X$
is a point contained in a cell of maximal dimension.

Assume now that $X$ is a stacky fan of pure dimension $n$.
The cells of dimension $n-1$ are called codimension one cells.
$X$ is said to be \emph{connected through codimension one} if for any two
maximal cells $X_k/G_k$ and $X_{k'}/G_{k'}$ one can find a sequence
of maximal cells $X_{k_0}/G_{k_0}=X_k/G_k$, $X_{k_1}/G_{k_1},\cdots,
X_{k_r}/G_{k_r}=X_{k'}/G_{k'}$ such that for any $0\leq i\leq r-1$
the two consecutive maximal cells $X_{k_i}/G_{k_i}$ and $X_{k_{i+1}}/G_{k_{i+1}}$
have a common codimension one cell in their closure.

\end{defi}

\begin{defi}\label{maps}
Let $X$ and $Y$ be two stacky fans with cells $\{X_k/G_k\}_{k\in K}$ and $\{Y_j/H_j\}_{j\in J}$ where
$\{X_k\subset \R^{m_k}\}_{k\in K}$ and $\{Y_j\subset \R^{m'_j}\}_{j\in J}$,
respectively.
A continuous map $\pi:X\to Y$ is said to be a \emph{map of stacky fans}
if for every cell $X_k/G_k$ of $X$ there exists a cell $Y_j/H_j$ of $Y$
such that
\begin{enumerate}
\item\label{map1} $\pi(X_k/G_k)\subset Y_j/H_j$;
\item\label{map2} $\pi: X_k/G_k \to Y_j/H_j$ is induced by an integral linear function
$L_{k,j}:\R^{m_k}\to \R^{m'_j}$, i.e. there exists a commutative diagram
\begin{equation}
\xymatrix{
X_k/G_k \ar[d]^{\pi}& X_k \ar@{^{(}->}[r] \ar@{->>}[l]
\ar[d]^{L_{k,j}}& \R^{m_k} \ar[d]^{L_{k,j}} \\
Y_j/H_j& Y_j \ar@{^{(}->}[r] \ar@{->>}[l]
& \R^{m_j'} .}
\end{equation}
\end{enumerate}
We say that $\pi: X\to Y$ is \emph{full} if it sends every cell $X_k/G_k$ of $X$
surjectively into some cell $Y_j/H_j$ of $Y$.
We say that $\pi:X\to Y$ is \emph{of degree one} if for every generic point
$Q\in Y_j/H_j\subset Y$ the inverse image $\pi^{-1}(Q)$ consists of a single point
$P\in X_k/G_k\subset X$ and the integral linear function $L_{k,j}$ inducing
$\pi: X_k/G_k\to Y_j/H_j$ is primitive (i.e. $L_{k,j}^{-1}(\Z^{m'_j})\subset \Z^{m_k}$).
\end{defi}

\begin{remark}\label{old-def}
The above definition of stacky fan is inspired by
some definitions of polyhedral complexes present
in the literature, most notably in \cite[Def. 5, 6]{KKMS},
\cite[Def. 2.12]{GM}, \cite[Def. 5.1]{AR} and \cite[Pag. 9]{GS}.

The notions of pure-dimension and connectedness through codimension one are
well-known in tropical geometry (see the Structure Theorem in \cite{McLS}).
\end{remark}



\subsection{Graphs}
\label{graphs}

Here we recall the basic notions of graph theory that we will need
in the sequel. We follow mostly the terminology and notations
of \cite{Die}.

Throughout this paper,
$\Gamma$  will be a finite connected graph. By finite we mean that $\Gamma$ has a finite number of vertices and edges; moreover loops or multiple edges are allowed. We denote by $V(\Gamma)$ the set of
vertices of $\Gamma$ and by $E(\Gamma)$ the set of edges of $\Gamma$. The \textit{valence} of a vertex $v$, $\val(v)$, is defined as the number of edges incident to $v$, with the usual convention that a loop around a vertex $v$ is counted twice
in the valence of $v$. A graph $\Gamma$ is $k$\textit{-regular} if $\val(v)=k$ for every $v\in V(\Gamma)$.
\begin{defi}\label{cycles&bonds}
A \textit{cycle} of $\Gamma$ is a subset $S\subseteq E(\Gamma)$ such that the graph $\Gamma/$ $(E(\Gamma)\setminus S)$,
obtained from $\Gamma$ by contracting all the edges not in $S$, is (connected and) $2$-regular.\\
If $\{V_1,V_2\}$ is a partition of $V(\Gamma)$, the set $E(V_1,V_2)$ of all the edges of $\Gamma$ with one end in $V_1$ and the other end in $V_2$ is called a \textit{cut}; a \textit{bond} is a minimal cut, or equivalently, a cut $E(\Gamma_1,\Gamma_2)$
such that the graphs $\Gamma_1$ and $\Gamma_2$ induced by $V_1$ and $V_2$, respectively, are connected.
\end{defi}
In the Example \ref{Exa-dual}, the subsets $\{f_1,f_2,f_3\}$ and $\{f_4,f_5\}$ are bonds of $\Gamma_2$ while
the subset $\{f_1,f_2,f_3,f_4,f_5\}$ is a non-minimal cut of $\Gamma_2$.

\begin{nota}{\emph{Homology theory}}\label{homo}

Consider the space of $1$-chains and $0$-chains of $\Gamma$ with values in a
finite abelian group $A$ (we will use the groups $A=\Z,\R$):
$$
  C_1(\Gamma, A):=\oplus_{e\in E(\Gamma)}A \cdot e \hspace{1,5cm}
C_0(\Gamma, A):=\oplus_{v\in V(\Gamma)} A\cdot v.
$$
We endow the above spaces with the $A$-bilinear, symmetric, non-degenerate forms
uniquely determined by:
$$
(e,e'):=\delta_{e,e'} \hspace{1,5cm} \langle v,v'\rangle :=\delta_{v,v'},
$$
where $\delta_{-,-}$ is the usual Kronecker symbol and $e, e'\in E(\Gamma)$;
$v, v'\in V(\Gamma)$. Given a subspace $V\subset C_1(\Gamma,A)$, we denote by $V^{\perp}$
the orthogonal subspace with respect to the form $(,)$.

Fix now an orientation of $\Gamma$ and let $s, t: E(\Gamma)\to V(\Gamma)$ be the two maps sending an
oriented edge to its source and target vertex, respectively. Define two boundaries maps
$$\begin{aligned}
\partial : C_1(\Gamma, A) & \longrightarrow C_0(\Gamma, A) & \hspace{1,5cm}
\delta : C_0(\Gamma, A) & \longrightarrow C_1(\Gamma, A) \\
e & \mapsto t(e)-s(e) & v & \mapsto \sum_{e\,:\, t(e)=v} e -\sum_{e\, :\, s(e)=v} e.
\end{aligned}$$
It is easy to check that the above two maps are adjoint with respects to
the two symmetric $A$-bilinear forms $(,)$ and $\langle,\rangle$, i.e.
$\langle \partial(e),v\rangle =(e, \delta(v))$ for any $e\in E(\Gamma)$ and $v\in V(\Gamma)$.

The kernel of $\partial$ is called the first homology group of $\Gamma$
with coefficients in $A$ and is denoted by $H_1(\Gamma, A)$. Since $\partial$ and $\delta$ are adjoint, it follows that
$H_1(\Gamma,A)^{\perp}=\Im(\delta).$
It is a well-known result in graph theory that $H_1(\Gamma,A)$ and $H_1(\Gamma,A)^{\perp}$ are free $A$-modules of ranks:
$$\begin{sis}
&\rm{rank}_{A} H_1(\Gamma, A)=1-\#V(\Gamma) +\#E(\Gamma), \\
&\rm{rank}_{A} H_1(\Gamma, A)^{\perp}=\#V(\Gamma)-1. \\
\end{sis}$$
The $A$-rank of $H_1(\Gamma,A)$ is called also the \emph{genus}
of $\Gamma$ and it is denoted by $g(\Gamma)$; the $A$-rank of
$H_1(\Gamma, A)^{\perp}$
is called the \emph{co-genus} of $\Gamma$ and it is denoted by $g^*(\Gamma)$.

\end{nota}

\begin{nota}{\emph{Connectivity and Girth}}\label{conn-girth}

There are two ways to measure the connectivity of a graph: the
vertex-connecti\-vi\-ty (or connectivity) and the edge-connectivity.
Recall their definitions (following \cite[Chap. 3]{Die}).

\begin{defi}\label{conn}
Let $k\geq 1$ be an integer.
\begin{enumerate}
\item A graph $\Gamma$
is said to be $k$-{\it vertex-connected} (or simply $k$-{\it connected}) if the
graph obtained from $\Gamma$ by removing any set of $s\leq k-1$ vertices and the edges adjacent to them
is connected.

\item The \textit{connectivity} of $\Gamma$, denoted by $k(\Gamma)$, is the maximum integer $k$
such that $\Gamma$ is $k$-connected. We set
$k(\Gamma)=+\infty$ if $\Gamma$ has only one vertex.

\item A graph $\Gamma$
is said to be $k$-{\it edge-connected}  if the graph obtained from $\Gamma$ by removing
any set of $s\leq k-1$ edges is connected.

\item The \textit{edge-connectivity} of $\Gamma$, denoted by $\lambda(\Gamma)$, is the
maximum integer $k$ such that $\Gamma$ is $k$-edge-connected. We set
$\lambda(\Gamma)=+\infty$ if $\Gamma$ has only one vertex.
\end{enumerate}
\end{defi}

Note that $\lambda(\Gamma)\geq 2$ if and only if $\Gamma$ has
no separating edges; while $\lambda(\Gamma)\geq 3$ if and only if
$\Gamma$ does not have pairs of separating edges.

In \cite{CV1}, a characterization of $3$-edge-connected
graphs is given in terms of the so-called {\it $C1$-sets}. Recall (see \cite[Def.
2.3.1, Lemma 2.3.2]{CV1}) that a $C1$-set of $\Gamma$ is a subset of
$E(\Gamma)$ formed by edges that are non-separating and belong to the same
cycles of $\Gamma$. The $C1$-sets form a
partition of the set of non-separating edges (\cite[Lemma 2.3.4]{CV1}). In \cite[Cor. 2.3.4]{CV1}, it is proved that $\Gamma$ is $3$-edge-connected if and
only if $\Gamma$ does not have separating edges and all the $C1$-sets have
cardinality one.

The two notions of connectivity are related by the following relation:
$$k(\Gamma)\leq \lambda(\Gamma)\leq \delta(\Gamma),$$
where $\delta(\Gamma):=\min_{v\in V(\Gamma)}\{\val(v)\}$ is the valence of $\Gamma$.


Finally recall the definition of the girth of a graph.

\begin{defi}\label{girth}
The \emph{girth} of a graph $\Gamma$, denoted by $\girth(\Gamma)$, is the minimum
integer $k$ such that $\Gamma$ contains a cycle of length $k$.
We set $\girth(\Gamma)=+\infty$ if $\Gamma$ has no cycles, i.e. if it is a tree.
\end{defi}

Note that $\girth(\Gamma)\geq 2$ if and only if $\Gamma$ has no loops;
while $\girth(\Gamma)\geq 3$ if and only if $\Gamma$ has no loops and no multiples
edges. Graphs with girth greater or equal than $3$ are called simple.

\begin{example}
For the graph $\Gamma_1$ in the Example \ref{Exa-dual}, we have that 
$k(\Gamma_1)=1$ because $v$ is a separating vertex. The $C1$-sets of $\Gamma_1$ are 
$\{e_1,e_2,e_3\}$ and $\{e_4,e_5\}$. We have that $\lambda(\Gamma_1)=2$ because
$\Gamma_1$ has a $C1$-set of cardinality greater than $1$ and does not have separating edges.
Moreover, $\girth(\Gamma_1)=2$ since $\{e_4,e_5\}$ is the smallest cycle of 
$\Gamma_1$.

The Peterson graph $\Gamma$ depicted in Figure \ref{peterson} is $3$-regular and has  
$k(\Gamma)=\lambda(\Gamma)=3$. Moreover, it is easy to check that 
$\girth(\Gamma)=5$.
\end{example}

\end{nota}

\begin{nota}{\emph{$2$-isomorphism}}

We introduce here an equivalence relation on the set of all graphs, that will
be very useful in the sequel.

\begin{defi}[\cite{Whi}]\label{2-iso}
Two graphs $\Gamma_1$ and $\Gamma_2$ are said to be \emph{$2$-isomorphic},
and we write $\Gamma_1\equiv_2 \Gamma_2$,
if there exists a bijection $\phi:E(\Gamma_1)\to E(\Gamma_2)$
inducing a bijection between cycles of $\Gamma_1$ and cycles of $\Gamma_2$, or equivalently, between bonds of $\Gamma_1$ and bonds of $\Gamma_2$.
We denote by $[\Gamma]_{\rm 2}$ the $2$-isomorphism class
of a graph $\Gamma$.
\end{defi}
This equivalence relation is called cyclic equivalence in
\cite{CV1} and denoted by $\equiv_{\rm cyc}$.

\begin{remark}
The girth, the connectivity, the edge-connectivity, the genus and the co-genus are
defined up to $2$-isomorphism; we denote them by $\girth([\Gamma]_2)$, $k([\Gamma]_2)$, $\lambda([\Gamma]_2)$, $g([\Gamma]_2)$ and $g^*([\Gamma]_2)$.
\end{remark}

As a consequence of a very well known theorem of Whitney (see \cite{Whi} or \cite[Sec. 5.3]{Oxl}),
we have the following

\begin{fact}
\label{3v} If $\Gamma$ is 3-connected,
the $2$-isomorphism class $[\Gamma]_2$ contains only $\Gamma$.

\end{fact}

In the sequel, graphs with girth or edge-connectivity at least $3$
will play an important role. We describe here a way to obtain such
a graph starting with an arbitrary graph $\Gamma$.

\begin{defi}\label{simpl}
Given a graph $\Gamma$, the {\it simplification} of $\Gamma$ is the
simple graph $\Gamma^{{\rm sim}}$ obtained from $\Gamma$ by deleting all
the loops and all but one among each collection of multiple edges.
\end{defi}
Note that the graph $\Gamma^{{\rm sim}}$ does not depend on the choices
made in the operation of deletion.
A similar operation can be performed with respect to the edge-connectivity,
but the result is only a $2$-isomorphism class of graphs.

\begin{defi}\cite[Def. 2.3.6]{CV1}\label{3-conn}
Given a graph $\Gamma$, a {\it 3-edge-connectivization} of $\Gamma$ is a
graph, denoted by $\Gamma^{3}$,
obtained from $\Gamma$ by contracting all the separating edges
and all but one among the edges of each $C1$-set of $\Gamma$.

The $2$-isomorphism class of $\Gamma^3$, which is independent of all the
choices made in the construction of $\Gamma^3$
(see \cite[Lemma 2.3.8(iii)]{CV1}), is called the $3$-edge-connectivization
class of $\Gamma$ and is denoted by $[\Gamma^3]_2$.
\end{defi}

\end{nota}

\begin{nota}{\emph{Duality}}

Recall the following definition (see \cite[Sec. 4.6]{Die}).

\begin{defi}\label{gr-dual}
Two graphs $\Gamma_1$ and $\Gamma_2$ are said to be in {\it abstract duality}
if there exists a bijection $\phi:E(\Gamma_1)\to E(\Gamma_2)$
inducing a bijection between cycles (resp. bonds) of $\Gamma_1$ and bonds (resp. cycles) of $\Gamma_2$.
Given a graph $\Gamma$, a graph $\Gamma'$ such that
$\Gamma$ and $\Gamma'$ are in abstract duality is called an abstract dual
of $\Gamma$ and is  denoted by $\Gamma^*$.
\end{defi}

\begin{example}\label{Exa-dual}
Let us consider the graphs
$$
\xymatrix@C=1pc@R=0.8pc{
& *{\bullet}\ar@{-}[dd]_{e_2} \ar@{-}[drr]^{e_1}& & & &\\
\Gamma_1 = & & &*{\bullet} \ar@/^/@{-}[rr]^{e_4}^<{v} \ar@/_/@{-}[rr]_{e_5}& &*{\bullet} \\
& *{\bullet} \ar@{-}[rru]_{e_3}& & & &
} 
\hspace{2cm}  
\xymatrix@C=1pc@R=0.8pc{
&&&& \\
\Gamma_2= & *{\bullet} \ar@/^/@{-}[rr]^{f_1} \ar@{-}[rr]|{f_2} \ar@/_/@{-}[rr]_{f_3}& & *{\bullet} 
\ar@/^/@{-}[rr]^{f_4}^<{w} \ar@/_/@{-}[rr]_{f_5}& & *{\bullet}  \\
}
$$


\noindent  The cycles of $\Gamma_1$ are $C_1:=\{e_1,e_2,e_3\}$ and $C_2:=\{e_4, e_5\}$, while  
the bonds of $\Gamma_2$ are $B_1:=\{f_1,f_2,f_3\}$ and $B_2:=\{f_4,f_5\}$.
The bijection $\phi:E(\Gamma_1)\to E(\Gamma_2)$ sending $e_i$ to $f_i$ for $i=1,\ldots,5$
sends the cycles of $\Gamma_1$ into the bonds of $\Gamma_2$; therefore $\Gamma_1$ and $\Gamma_2$ are in abstract 
duality.
\end{example}

Not every graph admits an abstract dual. Indeed we have the following
theorem of Whitney (see \cite[Theo. 4.6.3]{Die}).

\begin{thm}\label{exi-dual}
A graph $\Gamma$ has an abstract dual if and only if
$\Gamma$ is planar, i.e. if it
can be embedded into the plane.
\end{thm}
It is easy to give examples of planar graphs $\Gamma$ admitting
non-isomorphic abstract duals (see \cite[Example 2.3.6]{Oxl}).
However it follows easily from the definition that two abstract
duals of the same graph are $2$-isomorphic. Therefore, using
the above Theorem \ref{exi-dual}, it follows that abstract duality induces a
bijection
\begin{equation}\label{dual-iso}
\begin{aligned}
\left\{\text{Planar graphs}\right\}_{/\equiv_{2}} &\longleftrightarrow
\left\{\text{Planar graphs}\right\}_{/\equiv_{2}}\\
[\Gamma]_{2} & \longmapsto [\Gamma]_2^*:=[\Gamma^*]_{2}.
\end{aligned}
\end{equation}
Moreover, it is easy to check that the duality satisfies:
\begin{equation}\label{genus-dual}
\girth([\Gamma]_{2})=\lambda([\Gamma]_{2}^*) \hspace{1,2cm} g^*([\Gamma]_2)=g([\Gamma]_2^*) \hspace{1,2cm}
k([\Gamma]_2)=k([\Gamma]_2^*).
\end{equation}

\end{nota}

\subsection{Matroids}

Here we recall the basic notions of (unoriented) matroid theory that we will
need in the sequel. We follow mostly the terminology and notations
of \cite{Oxl}.

\begin{nota}{\emph{Basic definitions}}

There are several ways of defining a matroid (see \cite[Chap. 1]{Oxl}).
We will use the definition in terms of bases (see \cite[Sect. 1.2]{Oxl}).

\begin{defi}\label{matroid}
A {\it matroid} $M$ is a pair $(E(M), \BB(M))$ where $E(M)$ is a finite set, called the {\it ground set}, and $\BB(M)$ is a collection of subsets of $E(M)$, called {\it bases}
of $M$, satisfying the following two conditions:
\begin{enumerate}[(i)]
\item $\BB(M)\neq \emptyset$;
\item If $B_1,B_2\in \BB(M)$ and $x\in B_1\setminus B_2$, then there exists
an element $y\in B_2\setminus B_1$ such that $(B_1\setminus \{x\})\cup \{y\}
\in \BB(M)$.
\end{enumerate}

Given a matroid $M=(E(M), \BB(M))$, we define:
\begin{enumerate}[(a)]
\item The set of independent elements
$$\II(M):=\{I\subset E(M)\,:\, I\subset B \text{ for some } B\in \BB(M)\};$$
\item The set of dependent elements
$$\DD(M):=\{D\subset E(M)\,:\, E(M)\setminus D \in \II(M)\};$$
\item The set of circuits
$$\CC(M):=\{C\in \DD(M)\,:\, C \text{ is minimal among the elements of } \DD(M)\}.$$
\end{enumerate}
\end{defi}

It can be derived from the above axioms, that all the bases of $M$ have the
same cardinality, which is called the {\it rank} of $M$ and is denoted by $r(M)$.

Observe that each of the above sets $\BB(M)$, $\II(M)$, $\DD(M)$, $\CC(M)$
determines all the others. Indeed, it is possible to define a matroid $M$ in terms of the ground set $E(M)$ and each of the above sets, subject to suitable
axioms (see \cite[Sec. 1.1, 1.2]{Oxl}).

The above terminology comes from the following basic example of
matroids.

\begin{example}\label{rep-mat}
Let $F$ be a field and $A$ an $r\times n$ matrix of rank $r$ over $F$.
Consider the columns of $A$ as elements of the vector space $F^r$,
and call them $\{v_1, \ldots, v_n\}$.
The {\it vector matroid} of $A$, denoted by $M[A]$, is the matroid
whose ground set is $E(M[A]):=\{v_1,\ldots,v_n\}$ and whose bases are the
subsets of $E(M[A])$ consisting of vectors that form a base of
$F^r$. It follows easily that $\II(M[A])$ is formed by the subsets
of independent vectors of $E(M[A])$; $\DD(M[A])$
is formed by the subsets of dependent vectors and $\CC(M[A])$ is formed
by the minimal subsets of dependent vectors.
\end{example}

We now introduce a very important class of matroids.

\begin{defi}
A matroid $M$ is said to be {\it representable} over a field $F$, or
simply $F$-representable, if it is isomorphic to the vector matroid
of a matrix $A$ with coefficients in $F$. A matroid $M$ is said to be
{\it regular} if it is representable over any field $F$.
\end{defi}

Regular matroids are closely related to {\it totally unimodular matrices},
i.e. to real matrices for which every square submatrix has determinant
equal to $-1$, $0$ or $1$.
We say that two totally unimodular matrices $A, B\in M_{g,n}(\R)$
are {\it equivalent} if $A=XBY$ where $X\in {\rm GL}_g(\Z)$ and $Y\in {\rm GL}_n(\Z)$
is a permutation matrix.

\begin{thm}\label{reg-thm}
\begin{enumerate}[(i)]
\item A matroid $M$ of rank $r$ is regular if and only if $M=M[A]$ for a totally
unimodular matrix $A\in M_{g,n}(\R)$ of rank $r$, where $n=\#E(M)$
and $g$ is a
natural number such that $g\geq r$.
\item Given two totally unimodular matrices $A, B\in M_{g, n}(\R)$ of rank $r$,
we have that $M[A]=M[B]$ if and only if $A$ and $B$ are equivalent.
\end{enumerate}
\end{thm}
\begin{proof}
Part $(i)$ is proved in  \cite[Thm. 6.3.3]{Oxl}. Part $(ii)$ follows easily from
\cite[Prop. 6.3.13, Cor. 10.1.4]{Oxl}, taking into account that $\R$
does not have non-trivial automorphisms.
\end{proof}

In matroid theory, there is a natural duality theory (see
\cite[Chap. 2]{Oxl}).

\begin{defi}\label{mat-dual}
Given a matroid $M=(E(M), \BB(M))$, the {\it dual matroid} $M^*=(E(M^*), \BB(M^*))$
is defined by putting $E(M^*)=E(M)$ and
$$\BB(M^*)=\{B^*\subset E(M^*)=E(M)\,:\, E(M)\setminus B^* \in \BB(M)\}.$$
\end{defi}

It turns out that the dual of an $F$-representable matroid is again $F$-representable (see \cite[Cor. 2.2.9]{Oxl}) and therefore that
the dual of a regular matroid is again regular
(see \cite[Prop. 2.2.22]{Oxl}).

Finally, we need to recall the concept of simple matroid
(see \cite[Pag. 13, Pag. 52]{Oxl}).

\begin{defi}\label{simple-mat}
Let $M$ be a matroid. An element $e\in E(M)$ is called a {\it loop} if
$\{e\}\in \CC(M)$. Two distinct elements $f_1, f_2\in E(M)$ are called
{\it parallel} if $\{f_1, f_2\}\in \CC(M)$; a parallel class of $M$
is a maximal subset $X\subset E(M)$ with the property that
all the elements of $X$ are not loops and they are pairwise
parallel.

$M$ is called {\it simple} if it has no loops and all the parallel classes
have cardinality one.
\end{defi}

Given a matroid, there is a standard way to associate to it a simple
matroid.

\begin{defi}\label{simplifi-mat}
Let $M$ be a matroid. The {\it simple matroid associated} to $M$,
denoted by $\widetilde{M}$, is the matroid whose ground set
is obtained by deleting all the loops of $M$ and, for each
parallel class $X$ of $M$, deleting all but one distinguished
element of $X$ and whose set of bases is the natural one induced by
$M$.
\end{defi}

\end{nota}

\begin{nota}{\emph{Graphic and Cographic matroids}}

Given a graph $\Gamma$, there are two natural ways of associating a matroid to it.

\begin{defi}\label{grap&cogr-mat}
The {\it graphic matroid} (or {\it cycle matroid}) of $\Gamma$ is the matroid
$M(\Gamma)$ whose ground set is $E(\Gamma)$
and whose circuits are the cycles of $\Gamma$. The {\it cographic matroid} (or {\it bond matroid}) of $\Gamma$ is the matroid
$M^*(\Gamma)$ whose ground set is $E(\Gamma)$ and whose circuits
are the bonds of $\Gamma$.
\end{defi}

The rank of $M(\Gamma)$ is equal to $g^*(\Gamma)$ (see
\cite[Pag. 26]{Oxl}), and the rank of $M^*(\Gamma)$ is equal to $g(\Gamma)$,
as it follows easily from \cite[Formula 2.1.8]{Oxl}.

It turns out that $M(\Gamma)$ and $M^*(\Gamma)$ are regular matroids
(see \cite[Prop. 5.1.3, Prop. 2.2.22]{Oxl}) and that they are dual to each other
(see \cite[Sec. 2.3]{Oxl}).
Moreover we have the following obvious

\begin{remark}\label{2-isomo}
Two graphs $\Gamma_1$ and $\Gamma_2$ are $2$-isomorphic if and only if
$M(\Gamma_1)= M(\Gamma_2)$ or, equivalently, if and only if
$M^*(\Gamma_1)= M^*(\Gamma_2)$. Therefore, we can write
$M([\Gamma]_2)$ and $M^*([\Gamma]_2)$ for a $2$-isomorphism
class $[\Gamma]_2$.
\end{remark}

We have the following characterization of abstract dual graphs in terms
of matroid duality (see \cite[Sec. 5.2]{Oxl}).

\begin{prop}\label{2dual}
Let $\Gamma$ and $\Gamma^*$ be two graphs. The following conditions
are equivalent:
\begin{enumerate}[(i)]
\item $\Gamma$ and $\Gamma^*$ are in abstract duality;
\item $M(\Gamma)=M^*(\Gamma^*)$;
\item $M^*(\Gamma)=M(\Gamma^*)$.
\end{enumerate}
\end{prop}

By combining Proposition \ref{2dual} with Remark \ref{2-isomo}, we get the
following

\begin{remark}\label{gr-cogr}
There is a bijection between the following sets
$$\{\text{Graphic and cographic matroids}\}
\longleftrightarrow
\{\text{Planar graphs}\}_{/\equiv_2}.$$
Moreover this bijection is compatible with the respective duality theories,
namely the duality theory for matroids (definition \ref{mat-dual})
and the abstract duality theory for graphs (definition \ref{gr-dual}).
\end{remark}

Finally, we want to describe the simple matroid associated to a
graphic or to a cographic matroid, in terms of the simplification \ref{simpl}
and of the $3$-edge-connectiviza\-tion \ref{3-conn}.

\begin{prop}\label{simple-mat-graph}
Let $\Gamma$ be a graph. We have that
\begin{enumerate}[(i)]
\item $\widetilde{M(\Gamma)}=M(\Gamma^{{\rm sim}})$.
\item $\widetilde{M^*(\Gamma)}=M^*(\Gamma^{{\rm 3}})$, for any $3$-edge-connectivization $\Gamma^3$
of $\Gamma$.
\end{enumerate}
\end{prop}
\begin{proof}
The first assertion is well-known (see \cite[Pag. 52]{Oxl}).

The second assertion follows from the fact that an edge $e\in E(\Gamma)$
is a loop of $M^*(\Gamma)$ if and only if $e$ is a bond of $\Gamma$,
i.e. if $e$ is a separating edge of $\Gamma$;
and that a pair $f_1,f_2$  of edges is parallel in $M^*(\Gamma)$
if and only $\{f_1, f_2\}$
is a bond of $\Gamma$, i.e. if it is a pair of separating edges of $\Gamma$.
\end{proof}

\end{nota}

\section{The moduli space $\Mgt$}

\subsection{Tropical curves}\label{ps-trop}

In order to define tropical curves, we start with the
following

\begin{defi}
\label{mark-graph}
A \emph{marked graph} is a couple $(\Gamma,w)$ consisting of a finite connected
graph $\Gamma$ and a function $w:V(\Gamma)\to \N_{\geq 0}$, called
the weight function.
A marked graph is called \emph{stable} if any vertex $v$ of weight zero
(i.e. such that $w(v)=0$) has valence $\val(v)\geq3$.
The total weight of a marked graph $(\Gamma,w)$ is
$$|w|:=\sum_{v \in V(\Gamma)} w(v),$$
and the genus of $(\Gamma,w)$ is equal to
$$g(\Gamma,w):=g(\Gamma)+|w|.$$
We will denote by $\un{0}$ the identically zero weight function.
\end{defi}



\begin{remark}\label{finiteness}
It is easy to see that there is a finite number of stable marked
graphs of a given genus $g$.
\end{remark}

\begin{defi}\label{pseudo-trop}
A \emph{tropical curve} $C$ is the datum of a triple $(\Gamma,w,l)$
consisting of a stable marked graph $(\Gamma,w)$, called the
combinatorial type of $C$, and a function $l:E(\Gamma)\to \R_{>0}$,
called the length function.
The genus of $C$ is the genus of its combinatorial type.
\end{defi}

See \ref{Ex-Peterson} for an example of a tropical curve.


\begin{remark}\label{comp-MZ}
The above definition generalizes the definition
of (equivalence class of) tropical curves given by Mikhalkin-Zharkov in \cite[Prop. 3.6]{MZ}.
More precisely, tropical curves with total weight zero in our
sense are the same as
compact tropical curves up to tropical modifications
in the sense of Mikhalkin-Zharkov.
\end{remark}

A \emph{specialization} of a tropical curve
is obtained by letting some of its edge lengths  go to $0$, i.e. by contracting
some of its edges (see \cite[Sec.3.1.D]{MIK0}).
The weight function of the specialized curve changes according to the
following rule:
if we contract a loop $e$ around a vertex $v$ then we
increase the weight of $v$ by one; if we contract an edge $e$ between
two distinct vertices $v_1$ and $v_2$ then we obtain a new vertex
with weight equal to $w(v_1)+w(v_2)$.
We write $C\rightsquigarrow C'$ to denote that $C$ specializes to $C'$;
if $(\Gamma,w)$ (resp. $(\Gamma',w')$) are the combinatorial types
of $C$ (resp. $C'$), we write as well $(\Gamma,w)\rightsquigarrow (\Gamma',w')$.
Note that a specialization preserves the genus of the tropical
curves.




\subsection{Construction of $\Mgt$}\label{constr-Mgt}

Given a marked graph $(\Gamma,w)$, its automorphism group $\Aut(\Gamma,w)$
is the subgroup of $S_{|E(\Gamma)|}\times S_{|V(\Gamma)|}$
consisting of all pairs of permutations $(\phi,\psi)$ such that
$w(\psi(v))=w(v)$ for any $v\in V(\Gamma)$
and, for a fixed orientation of $\Gamma$,
we have that $\{s(\phi(e)), t(\phi(e))\}=\{\psi(s(e)), \psi(t(e))\}$
for any $e\in E(\Gamma)$, where $s, t: E(\Gamma)\to V(\Gamma)$ are the source
and target maps corresponding to the chosen orientation. Note that this definition is independent of the orientation. There is a natural homomorphism
$$\rho_{(\Gamma,w)}: \Aut(\Gamma,w)\to S_{|E(\Gamma)|}
\subset GL_{|E(\Gamma)|}(\Z)$$
induced by the projection of
$\Aut(\Gamma,w)\subset S_{|E(\Gamma)|}\times S_{|V(\Gamma)|}$
onto the second factor followed by the inclusion of $S_{|E(\Gamma)|}$
into $GL_{|E(\Gamma)|}(\Z)$ as the subgroup of the permutation matrices.

The group $\Aut(\Gamma,w)$ acts on $\R^{|E(\Gamma)|}$ via the
homomorphism $\rho_{(\Gamma,w)}$ preserving the open
rational polyhedral cone
$\R_{>0}^{|E(\Gamma)|}$ and its closure
$\R_{\geq 0}^{|E(\Gamma)|}$. We denote the respective quotients by
$$C(\Gamma,w):=\R_{>0}^{|E(\Gamma)|}/\Aut(\Gamma,w)
\hspace{0.5cm} \text{ and } \hspace{0.5cm}
\ov{C(\Gamma,w)}:=\R_{\geq 0}^{|E(\Gamma)|}/\Aut(\Gamma,w)
$$
endowed with the quotient topology.
When $\Gamma$ is such that $E(\Gamma)=\emptyset$ and $V(\Gamma)$ is just one vertex of weight $g$, we set $C(\Gamma,w):=\{0\}$.
Note that $C(\Gamma,w)$
parametrizes tropical curves of combinatorial type equal
to $(\Gamma,w)$.


Observe that, for any specialization $i: (\Gamma,w) \rightsquigarrow
(\Gamma',w')$, we get a natural continuous map
$$\ov i: \R_{\geq 0}^{|E(\Gamma')|} \hookrightarrow \R_{\geq 0}^{|E(\Gamma)|}
\twoheadrightarrow \ov{C(\Gamma,w)},$$
where  $\ov{C(\Gamma,w)}$ is endowed with the quotient topology.
Note that, if $i$ is a nontrivial specialization, the image of the map $\ov i$ is contained in $\ov{C(\Gamma,w)}\setminus C(\Gamma,w)$,
so it does not meet the locus of $\ov{C(\Gamma,w)}$ parametrizing tropical curves of combinatorial type $(\Gamma,w)$.



We are now ready to define the moduli space of tropical curves of fixed genus.

\begin{defi}\label{Mgt}
We define $\Mgt$ as the topological space (with respect to the quotient
topology)
$$\Mgt:=\left(\coprod \ov{C(\Gamma,w)}\right)_{/ \sim}
$$
where the disjoint union (endowed with the disjoint union topology)
runs through all stable marked graphs $(\Gamma,w)$ of genus $g$
and $\sim$ is the equivalence relation generated by the following
binary relation $\approx$:
given two points $p_1\in \ov{C(\Gamma_1,w_1)}$ and $p_2\in \ov{C(\Gamma_2,w_2)}$,
$p_1\approx p_2$
iff  there exists
a stable marked graph $(\Gamma,w)$ of genus $g$, a point
$q\in \R_{\geq 0}^{|E(\Gamma)|}$  and two specializations
$i_1:(\Gamma_1,w_1)\rightsquigarrow (\Gamma,w)$
and $i_2:(\Gamma_2,w_2)\rightsquigarrow (\Gamma,w)$  such that
$\ov i_1(q)=p_1$ and $\ov i_2(q)=p_2$.
\end{defi}

From the definition of the above equivalence relation $\sim$, we get  the following

\begin{remark}\label{Mgtrem}
\noindent
\begin{itemize}
\item[(i)] Let $p_1,p_2\in \coprod \ov{C(\Gamma,w)}$ such that $p_1\sim p_2$.
If  there exist two stable marked graphs $(\Gamma_1,w_1)$ and $(\Gamma_2,w_2)$ such that $p_1\in C(\Gamma_1,w_1)$ and $p_2\in C(\Gamma_2,w_2)$,
then $(\Gamma_1,w_1)=(\Gamma_2,w_2)$ and $p_1=p_2$.
\item[(ii)] Let $p\in\coprod \ov{C(\Gamma,w)}$. Then there exists a stable marked graph $(\Gamma',w')$
and $p'\in C(\Gamma',w')$ such that $p\sim p'$.
\end{itemize}
\end{remark}






\begin{example}\label{genus2}
In the following figure we represent all stable marked graphs corresponding to tropical curves of genus $2$. The arrows represent all possible specializations.

\begin{figure}[h]
\begin{equation*}
\xymatrix@=.5pc{
 *{\bullet} \ar @{-} @/_.7pc/[rr] \ar@{-} @/^.7pc/[rr]
\ar @{-}[rr]^<{0}^>{0} && *{\bullet} & & & &&
*{\bullet} \ar@{-}@(ul,dl) \ar@{-}[rr]^<{0}^>{0} && *{\bullet}
\ar@{-}@(dr,ur) &\\
 & \ar @{-->}_{i_1}[ddd] &  &&&&& \ar @{-->}^{i_2}[dddlllll]  & \ar @{-->}[ddd] &
\\
\\
\\
&&&&&&&&&&&&\\
&*{\bullet} \ar@{-}_<{0}@(ul,dl) \ar@{-}@(dr,ur) &&& && &*{\bullet} \ar@{-}@(ul,dl) \ar@{-}[rr]^>{1}^<{0}&& *{\bullet} & \\
& \ar @{-->}[ddd] &   &&&&&   \ar @{-->}[dddlllll] & \ar @{-->}[ddd]
 &&&&& \\
\\
\\
&&&&&&&&&&&&&\\
&&*{\bullet\, 1} \ar@{-}@(ul,dl) &&&&& *{\bullet} \ar@{-}[rr]^<{1}^>{1} && *{\bullet}\\
&& \ar @{-->}[dddrrr]  &&&&&&  \ar @{-->}[dddlll]
\\
\\
&&&&&&&&&&&&&\\
&&&&& {\bullet\, 2}
}
\end{equation*}
\caption{Specializations of tropical curves of genus $2$.}\label{genus2fig}
\end{figure}
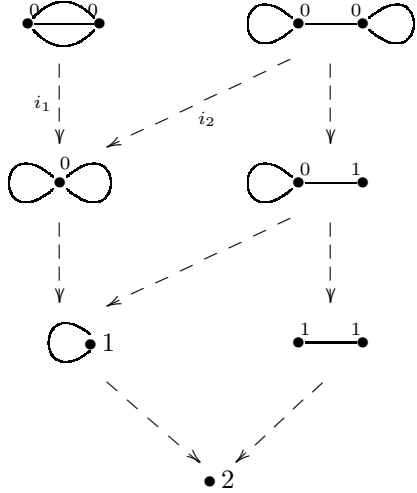

The cells corresponding to the two graphs on the top of Figure \ref{genus2fig}
are $\R^3_{\geq 0}/S_3$ and $\R^3_{\geq 0}/S_2$, respectively.
According to Definition \ref{Mgt},
$M_2^{\rm tr}$ corresponds to the
topological space obtained by gluing $\R^3_{\geq 0}/S_3$ and $\R^3_{\geq 0}/S_2$
along the points of $(\R^3_{\geq 0}/S_3) \setminus (\R^3_{>0}/S_3)$ and of
$(\R^3_{\geq 0}/S_2) \setminus (\R^3_{>0}/S_3)$ that correspond to common specializations
of those graphs according to the above diagram.
For instance, the specializations $i_1$ and $i_2$ induce the maps
\begin{center}
\begin{tabular}{rccccccrccc}
$\ov i_1:$&$\R^2_{\geq 0}$ &$\to$& $\R^3_{\geq 0}/S_3$& &and& &$\ov i_2:$&$\R^2_{\geq 0}$ &$\to$& $\R^3_{\geq 0}/S_2$,\\
&$(a_1,a_2)$& $\mapsto$ & $[(a_1,a_2,0)]$
&&&&&
$(a_1,a_2)$& $\mapsto$ & $[(a_1,0,a_2)]$
\end{tabular}
\end{center}
where in $\R^3_{\geq 0}/S_2$ the second coordinate
corresponds to the edge of the graph connecting the two vertices.
So, a point $[(x_1,x_2,x_3)]\in \R^3_{\geq 0}/S_3$ will be identified with a point
$[(y_1,y_2,y_3)]\in \R^3_{\geq 0}/S_2$ via the maps $\ov i_1$ and $\ov i_2$
if $y_2=0$ and if there exists $\sigma \in S_3$ such that
$(y_1,y_3)=(x_{\sigma(1)},x_{\sigma(2)})$ and $x_{\sigma(3)}=0$.
\end{example}

\begin{thm}\label{Mgt-pol}
The topological space $\Mgt$ is a stacky fan with cells
$C(\Gamma,w)$, as $(\Gamma,w)$ varies through all
stable marked graphs of genus $g$. In particular, its points are in bijection with
tropical curves of genus $g$.
\end{thm}
\begin{proof}
 Let us prove the first statement, by checking the conditions
of Definition \ref{pol-compl}. Consider the maps
$\alpha_{(\Gamma,w)}:\ov{C(\Gamma,w)}\to \Mgt$ naturally
induced by $\ov{C(\Gamma,w)}\hookrightarrow \coprod \ov{C(\Gamma',w')}
\twoheadrightarrow \Mgt$. The maps $\alpha_{(\Gamma,w)}$ are continuous by
definition of the quotient topology and the restriction of $\alpha_{(\Gamma,w)}$
to $C(\Gamma,w)$ is a bijection onto its image by Remark \ref{Mgtrem}(i).
Moreover, given an open
subset $U\subseteq C(\Gamma,w)$, $\alpha_{(\Gamma,w)}(U)$ is an open subset of $\Mgt$
since its inverse image on $\coprod \ov{C(\Gamma',w')}$ is equal to $U$.
This proves that the maps  $\alpha_{(\Gamma,w)}$ when restricted to $C(\Gamma,w)$ are homeomorphisms
onto their images,
and condition \ref{pol-compl}(\ref{pol1}) is satisfied.

From Remark \ref{Mgtrem}(ii), we get that
\begin{equation}\label{disj-union}
\Mgt=\bigcup_{(\Gamma,w)} \alpha_{(\Gamma,w)}(C(\Gamma,w))
\end{equation}
and the union is disjoint by Remark \ref{Mgtrem}(i);
thus condition \ref{pol-compl}(\ref{pol2}) is satisfied.

Let us check the condition
\ref{pol-compl}(\ref{pol3}). Let $(\Gamma,w)$ and $(\Gamma',w')$ be two stable marked graphs of genus $g$
and set $\alpha:=\alpha_{(\Gamma,w)}$ and $\alpha':=\alpha_{(\Gamma',w')}$.
By definition of $\Mgt$, the intersection of the images of $\ov{C(\Gamma,w)}$
and $\ov{C(\Gamma',w')}$ in $\Mgt$ is equal to
$$\alpha (\ov{C(\Gamma,w)})\cap
\alpha' (\ov{C(\Gamma',w')})
=\coprod_i \alpha_i(C(\Gamma_i,w_i)),$$
where $(\Gamma_i,w_i)$ runs over all common specializations of $(\Gamma,w)$ and $(\Gamma',w')$.
We have to find an integral linear map $L:\R^{|E(\Gamma)|}\to \R^{|E(\Gamma')|}$
making the following diagram commutative
\begin{equation}\label{lin-diag}
\xymatrix{
\coprod_i \alpha_i (C(\Gamma_i,w_i)) \ar@{^{(}->}[rd] \ar@{^{(}->}[r]&
\alpha (\ov{C(\Gamma,w)})& \R_{\geq 0}^{|E(\Gamma|} \ar@{^{(}->}[r] \ar@{->>}[l]
\ar[d]^L& \R^{|E(\Gamma)|} \ar[d]^{L} \\
&\alpha'(\ov{C(\Gamma',w')})& \R_{\geq 0}^{|E(\Gamma')|} \ar@{^{(}->}[r] \ar@{->>}[l]
& \R^{|E(\Gamma')|} .}
\end{equation}
To this aim, observe that, since $(\Gamma_i, w_i)$ are specializations of both $(\Gamma,w)$ and
$(\Gamma',w')$, there are orthogonal projections $f_i:\R^{|E(\Gamma)|}\twoheadrightarrow \R^{|E(\Gamma_i)|}$
and inclusions $g_i: \R^{|E(\Gamma_i)|} \hookrightarrow \R^{|E(\Gamma')|}$.
 We define $L$ as the composition
$$L:\R^{|E(\Gamma)|}\stackrel{\oplus f_i}{\longrightarrow} \oplus_i
\R^{|(E(\Gamma_i)|} \stackrel{\oplus g_i}{\longrightarrow}
\R^{|E(\Gamma)|}.$$
It is easy to see that $L$ is an integral linear map making
the above diagram (\ref{lin-diag}) commutative, and this concludes the proof
of the first statement.

The second statement follows from (\ref{disj-union}) and the fact,
already observed before, that $C(\Gamma,w)$ parametrizes
tropical curves of combinatorial type $(\Gamma,w)$.
\end{proof}

We now prove that $\Mgt$ is of pure dimension and connected through codimension one.
To that aim, we describe the maximal cells and the codimension one cells of $\Mgt$.

\begin{prop}\label{prop-Mgt}
\noindent
\begin{enumerate}[(i)]
\item\label{Mgt-pure}
The maximal cells of $\Mgt$ are exactly those of the form $C(\Gamma,\un{0})$ where
$\Gamma$ is $3$-regular. In particular,
$\Mgt$ is of pure dimension $3g-3$.
\item \label{Mgt-conn} $\Mgt$ is connected through codimension one.
\item \label{Mgt-codim1}
The codimension one cells of $\Mgt$ are of the following two types:
\begin{enumerate}
\item[(a)] $C(\Gamma,\un{0})$ where  $\Gamma$ has exactly one vertex
of valence $4$ and all other vertices of valence $3$;
\item[(b)] $C(\Gamma,w)$ where $\Gamma$ has exactly one vertex $v$ of
valence $1$ and weight $1$, and all the other vertices of valence $3$ and weight
$0$.
\end{enumerate}
Each codimension one cell of type $(b)$ lies in the closure of exactly one
maximal cell, while each codimension one cell of type $(a)$ lies in
the closure of one, two or three maximal cells.
\end{enumerate}
\end{prop}
\begin{proof}
First of all, observe that given a stable marked graph $(\Gamma,w)$
of genus $g$ we have
\begin{equation}\label{inequ-dim}
3|V(\Gamma)|\leq \sum_{v\in V(\Gamma)}\left[\val(v)+2w(v)\right]=2|E(\Gamma)|+2|w|,
\end{equation}
and the equality holds if and only if every $v\in V(\Gamma)$ is such that
either $w(v)=0$ and $\val(v)=3$ or $w(v)=\val(v)=1$. By substituting the formula
for the genus $g=g(\Gamma,w)=g(\Gamma)+|w|=1+|E(\Gamma)|-|V(\Gamma)|+|w|$ in inequality (\ref{inequ-dim}), we obtain
\begin{equation}\label{max-dim}
|E(\Gamma)|\leq 3g-3-|w|.
\end{equation}

Let us now prove part (\ref{Mgt-pure}). If $\Gamma$ is $3$-regular and $w\equiv 0$,
then $g(\Gamma)=g(\Gamma,w)=g$ and an easy calculation gives that $|E(\Gamma)|=3g-3$.
Therefore $\dim (C(\Gamma,\un{0}))=3g-3$, which is the maximal possible dimension of the cells
of $\Mgt$ according to the above inequality (\ref{max-dim}). Hence $C(\Gamma,\un{0})$
is maximal. On the other hand, every stable marked graph $(\Gamma',w')$ can be obtained
by specializing a stable marked graph $(\Gamma,\un{0})$ with
$\Gamma$ a $3$-regular graph (see for example \cite[Appendix A.2]{CV1}),
which concludes the proof of part (\ref{Mgt-pure}).

Let us prove part (\ref{Mgt-conn}). It is well-known
(see the appendix of \cite{HT} for a topological proof, \cite[Thm. II]{Tsu}
for a combinatorial proof in the case of simple graphs and \cite[Thm 3.3]{Cap}
for a combinatorial proof in the general case)
that any two $3$-regular graphs $\Gamma_1$ and $\Gamma_2$ of genus $g$
can be obtained one from the other via a sequence of twisting operations
as the one shown in the top line of Figure \ref{twist-3reg} below.
In each of these twisting operations,
the two graphs $\Gamma_1$ and $\Gamma_2$ specialize to a common graph $\Gamma$
(see Figure \ref{twist-3reg}) that has one vertex of valence $4$ and all the others
of valence $3$. By what will be proved below, $C(\Gamma,\un{0})$ is a codimension one cell.
 Therefore the two maximal dimensional cells $C(\Gamma_1,\un{0})$ and $C(\Gamma_2,\un{0})$ contain
 a common codimension one cell $C(\Gamma,\un{0})$ in their closures, which concludes
the proof of part (\ref{Mgt-conn}).
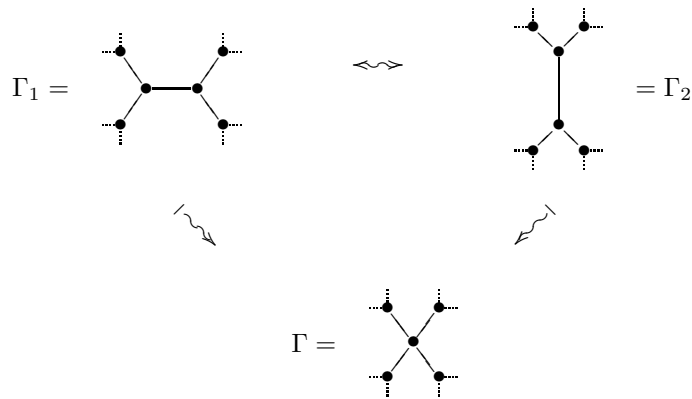
\begin{figure}[h]
\begin{tabular}{ccc}
\xymatrix@=0,3pc{
\\
&& \ar @{..}[d] &&&&\ar @{..}[d]\\
&& *{\bullet} \ar @{-}[dr] \ar @{..}[l] &&&& *{\bullet} \ar @{-}[dl] \ar @{..}[r] &\\
\Gamma_1=&&& *{\bullet} \ar @{-}[dl] \ar @{-}[rr] && *{\bullet}
\ar @{-}[dr] \\
&& *{\bullet} \ar @{..}[l] &&&& *{\bullet}  \ar @{..}[r]&\\
&& \ar @{..}[u] &&&& \ar @{..}[u]\\
\\
&&&&\ar @{|~>}[ddrr]\\
\\
&&&&&&
}

&
\xymatrix@=0,5pc{
\\
\\
\ar @{<~>}[rr]&&
}

&
\xymatrix@=0,3pc{
 &\ar @{..}[d] &&\ar @{..}[d]\\
\ar @{..}[r]& *{\bullet} \ar @{-}[dr]  && *{\bullet} \ar @{-}[dl] \ar @{..}[r] &\\
&& *{\bullet} \ar @{-}[dd] \\
&&&&& =\Gamma_2 \\
&& *{\bullet} \ar @{-}[dl] \ar @{-}[dr]\\
& *{\bullet} \ar @{..}[l] && *{\bullet}  \ar @{..}[r]&\\
& \ar @{..}[u] && \ar @{..}[u]\\
&&\ar @{|~>}[ddll]\\
\\
&&&&
}\\

& \xymatrix@=0,3pc{
&& \ar @{..}[d] &&\ar @{..}[d]\\
&\ar @{..}[r]& *{\bullet} \ar @{-}[dr]  && *{\bullet} \ar @{-}[dl] \ar @{..}[r] &\\
\Gamma=&&& *{\bullet} \ar @{-}[dl] \ar @{-}[dr]\\
&& *{\bullet} \ar @{..}[l] && *{\bullet}  \ar @{..}[r]&\\
&& \ar @{..}[u] && \ar @{..}[u]\\
}
\end{tabular}

\caption{The $3$-regular graphs $\Gamma_1$ and $\Gamma_2$ are twisted. They both specialize
to $\Gamma$. $C(\Gamma_1, \un{0})$ and $C(\Gamma_2,\un{0})$ are maximal dimensional
cells containing the codimension one cell $C(\Gamma,\un{0})$ in their closures.}\label{twist-3reg}
\end{figure}

Let us prove part (\ref{Mgt-codim1}). Let $C(\Gamma,w)$ be a codimension one cell
of $\Mgt$, i.e. such that $|E(\Gamma)|=3g-4$. According to the inequality
(\ref{max-dim}), there are two possibilities: either $|w|=0$ or $|w|=1$.
In the first case, i.e. $|w|=0$, using the inequality in
(\ref{inequ-dim}), it is easy to check that there should exist exactly
one vertex $v$ such that $\val(v)=4$ and all the other vertices
should have valence equal to $3$, i.e. we are in case (a).
In the second case, i.e. $|w|=1$, all the inequalities in (\ref{inequ-dim})
should be equalities and this implies that there should be exactly
one vertex $v$ such that $\val(v)=w(v)=1$ and all the other vertices have weight
equal to zero and valence equal to $3$, i.e. we are in case (b).

For a codimension one cell of type (a), $C(\Gamma,\un{0})$, there can be at most three
maximal cells $C(\Gamma_i,\un{0})$ ($i=1, 2, 3$) containing it in their closures,
as we can see in Figure \ref{spec-a}.
Note, however, that it can happen
that some of the $\Gamma_i$'s are isomorphic, and in that case the number of maximal
cells containing $C(\Gamma,\un{0})$ in their closure is strictly smaller than $3$.

\begin{figure}[h]
\begin{tabular}{ccc}
\xymatrix@=0,3pc{
\\
&& *{\bullet} \ar @{-}^<{1}[dr] &&&& *{\bullet} \ar @{-}_<{3}[dl]  &\\
\Gamma_1=&&& *{\bullet} \ar @{-}^>{2}[dl]\ar @{-}[rr] && *{\bullet}\ar @{-}_>{4}[dr] \\
&& *{\bullet}&&&& *{\bullet} &\\
\\
&&&&\ar @{|~>}[ddrr]\\
\\
&&&&&&
}
&
\xymatrix@=0,3pc{
\\
&& *{\bullet} \ar @{-}^<{1}[dr] &&&& *{\bullet} \ar @{-}_<{2}[dl]  &\\
\Gamma_2=&&& *{\bullet} \ar @{-}^>{3}[dl]\ar @{-}[rr] && *{\bullet}\ar @{-}_>{4}[dr] \\
&& *{\bullet}&&&& *{\bullet} &\\
\\
&&&&\ar @{|~>}[dd]\\
\\
&&&&&&
}
&
\xymatrix@=0,3pc{
\\
&& *{\bullet} \ar @{-}^<{1}[dr] &&&& *{\bullet} \ar @{-}_<{2}[dl]  &\\
\Gamma_3=&&& *{\bullet} \ar @{-}^>{4}[dl]\ar @{-}[rr] && *{\bullet}\ar @{-}_>{3}[dr] \\
&& *{\bullet}&&&& *{\bullet} &\\
\\
&&&&\ar @{|~>}[ddll]\\
\\
&&&&&&
}
\\
& \xymatrix@=0,3pc{
&& *{\bullet} \ar @{-}_<{1}[dr]  && *{\bullet} \ar@{-}^<{3}[dl] &\\
\Gamma=&&& *{\bullet} \ar @{-}_>{2}[dl] \ar@{-}^>{4}[dr]\\
&& *{\bullet} && *{\bullet}  &\\
}
\end{tabular}

\caption{The codimension one cell $C(\Gamma,\un{0})$ is contained in the closure
of the three maximal cells $C(\Gamma_i,\un{0})$, $i=1, 2, 3$.}\label{spec-a}
\end{figure}
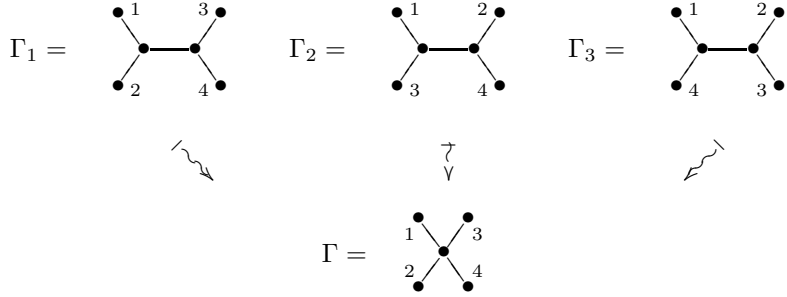

For a codimension one cell $C(\Gamma, w)$ of
type (b), there is only one maximal cell $C(\Gamma',\un{0})$ containing it in its
closure, as we can see in Figure \ref{spec-b} below.
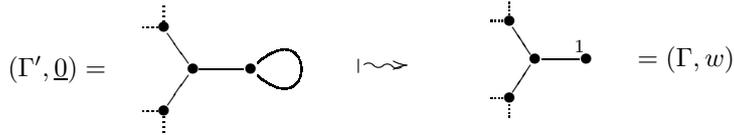
\begin{figure}[h]
\begin{tabular}{ccc}
\xymatrix@=0,4pc{
\\
&& \ar @{..}[d] &&&&\\
&& *{\bullet} \ar @{-}[dr] \ar @{..}[l] &&&& &\\
(\Gamma',\un{0})=&&& *{\bullet} \ar @{-}[dl] \ar @{-}[rr] && *{\bullet} \ar@{-}@(ur,dr) \\
&& *{\bullet} \ar @{..}[l] &&&&  \\
&& \ar @{..}[u] &&&& \\
\\
&&&&&&
}
&
\xymatrix@=0,5pc{
\\
\\
\\
\ar @{|~>}[rr]&&
}
&
\xymatrix@=0,3pc{
\\
&& \ar @{..}[d] &&&&\\
&& *{\bullet} \ar @{-}[dr] \ar @{..}[l] &&&& &\\
&&& *{\bullet} \ar @{-}[dl] \ar@{-}^>{1}[rr] && *{\bullet} &&=(\Gamma,w)\\
&& *{\bullet} \ar @{..}[l] &&&&  \\
&& \ar @{..}[u] &&&& \\
\\
&&&&&&
}
\end{tabular}
\caption{The codimension one cell $C(\Gamma,w)$ is contained
in the closure of the maximal dimensional cell
$C(\Gamma',\un{0})$.}\label{spec-b}
\end{figure}
\end{proof}


\section{The moduli space $\Agt$}

\subsection{Tropical abelian varieties}

\begin{defi}\label{ps-abvar}
A principally polarized \emph{tropical abelian variety} $A$ of dimension
$g$ is a $g$-dimensional real torus $\R^g/\Lambda$, where $\Lambda$ is a lattice
of rank $g$ in $\R^g$ endowed with a flat semi-metric
induced by a positive semi-definite quadratic form $Q$ on $\R^g$ such that the null space
$\Null(Q)$ of $Q$ is defined over $\Lambda\otimes \Q$, i.e. it admits a basis with elements in $\Lambda\otimes \Q$. Two tropical abelian varieties
$(\R^g/\Lambda,Q)$ and $(\R^g/\Lambda',Q')$ are isomorphic if there exists
$h\in GL(g,\R)$ such that $h(\Lambda)=\Lambda'$ and $hQh^t=Q'$.
\end{defi}

From now on, we will drop the attribute principally polarized as all the
tropical abelian varieties that we will consider are of this kind.

\begin{remark}\label{comp-MK2}
The above definition generalizes the definition of tropical
abelian variety given by Mikhalkin-Zharkov in \cite[Sec. 5]{MZ}.
More precisely, tropical abelian varieties endowed with positive definite
quadratic forms in our sense are the same as
(principally polarized) tropical abelian varieties
in the sense of Mikhalkin-Zharkov.
\end{remark}

\begin{remark}\label{aritequiv-quadr}
Every tropical abelian variety $(\R^g/\Lambda,Q)$ can be written in the form
$(\R^g/\Z^g,Q')$. In fact, it is enough to consider $Q'=hQh^t$, where
$h\in GL(g,\R)$ is such that $h(\Lambda)=\Z^g$.
Moreover, $(\R^g/\Z^g,Q)\cong (\R^g/\Z^g,Q')$ if and
only if there exists $h\in \GL_g(\Z)$ such that $Q'=h Q h^t$, i.e., if and only if
$Q$ and $Q'$ are arithmetically equivalent.
Therefore, from now on we will always consider our tropical abelian varieties
in the form $(\R^g/\Z^g,Q)$, where $Q$ is uniquely defined up to arithmetic equivalence.
\end{remark}

\subsection{Definition of $\Agt$ and $\Agcom$}

Let us denote by $\R^{\binom{g+1}{2}}$ the vector space of quadratic forms
in $\R^g$ (identified with $g\times g$ symmetric matrices with
coefficients in $\R$), by
$\O$ the cone in $\R^{\binom{g+1}{2}}$ of positive
definite quadratic forms and by $\Ort$ the cone of positive semi-definite quadratic
forms with rational null space (the so-called
rational closure of $\O$, see \cite[Sec. 8]{NamT}).

The group $\GL_g(\Z)$ acts on $\R^{\binom{g+1}{2}}$
via the usual law $h\cdot Q:= h Q h^t$, where $h\in \GL_g(\Z)$ and
$Q$ is a quadratic form on $\R^g$. This action naturally defines a homomorphism
$\rho:\GL_g(\Z)\to \GL_{\binom{g+1}{2}}(\Z)$.
Note that the cones $\O$ or $\Ort$ are preserved by the action of $\GL_g(\Z)$.

\begin{remark}\label{rat-qua}
It is well-known (see \cite[Sec. 8]{NamT}) that a positive semi-definite
quadratic form $Q$ in $\R^g$ belongs to $\Ort$ if and only if
there exists $h\in \GL_g(\Z)$ such that
$$hQh^t=
\left(\begin{array}{cc}
Q' & 0 \\
0 &  0 \\
\end{array}\right)
$$
for some positive definite quadratic form $Q'$ in $\R^{g'}$, with $0\leq g'\leq g$.
\end{remark}

\begin{defi}\label{Agtr}
We define $\Agt$ as the topological space (with respect to the quotient
topology)
$$\Agt:=\Ort/\GL_g(\Z).
$$
\end{defi}

The space $\Agt$ parametrizes tropical abelian varieties as it follows
from Remark \ref{aritequiv-quadr}. However, in order to endow $\Agt$ with the
structure of stacky fan, we need to specify some
extra-data, encoded
in the following definition (see \cite[Lemma 8.3]{NamT} or
\cite[Chap. IV.2]{FC}).

\begin{defi}\label{decompo}
A $\GL_g(\Z)$-\emph{admissible decomposition} of $\Ort$ is a collection $\Sigma=\{\sigma_{\mu}\}$ of rational polyhedral cones of
$\Ort$ such that:
\begin{enumerate}
\item If $\sigma$ is a face of $\sigma_{\mu}\in \Sigma$ then $\sigma\in \Sigma$;
\item The intersection of two cones $\sigma_{\mu}$ and $\sigma_{\nu}$ of $\Sigma$ is a face of both cones;
\item If $\sigma_{\mu}\in \Sigma$ and $h\in \GL_g(\Z)$ then $h\cdot \sigma_{\mu}
\cdot h^t\in \Sigma$.
\item $\#\{\sigma_{\mu}\in \Sigma \mod \GL_g(\Z)\}$ is finite;
\item $\cup_{\sigma_{\mu}\in \Sigma} \sigma_{\mu}=\Ort$.
\end{enumerate}
\end{defi}

Each $\GL_g(\Z)$-admissible decomposition of $\Ort$ gives rise to a structure of
stacky fan on $\Agt$.
In order to prove that, we need first to set some notations.

Let $\Sigma=\{\sigma_{\mu}\}$ be a $\GL_g(\Z)$-admissible decomposition
of $\Ort$. For each $\sigma_{\mu}\in \Sigma$ we set
$\sigma_{\mu}^0:=\Int(\sigma_{\mu})$; we denote by $\langle \sigma_{\mu}\rangle$ the smallest
linear subspace of $\R^{\binom{g+1}{2}}$ containing $\sigma_{\mu}$ and we set
$m_{\mu}:=\dim_{\R}\langle \sigma_{\mu}\rangle$.
Consider the stabilizer of $\sigma_{\mu}^0$ inside $\GL_g(\Z)$
$$\Stab(\sigma_{\mu}^0):=\{h\in \GL_g(\Z)\: : \: \rho(h)\cdot \sigma_{\mu}^0=
h\cdot \sigma_{\mu}^0
\cdot h^t=\sigma_{\mu}^0\}.$$
The restriction of the homomorphism $\rho$ to $\Stab(\sigma_{\mu}^0)$ defines
a homomorphism
$$\rho_{\mu}:\Stab(\sigma_{\mu}^0)\to \GL(\langle \sigma_{\mu}\rangle, \Z)=\GL_{m_{\mu}}(\Z).$$
By definition, the image $\rho_{\mu}(\Stab(\sigma_{\mu}^0))$ acts on $\langle \sigma_{\mu}\rangle=
\R^{m_\mu}$
and stabilizes the cone $\sigma_{\mu}^0$, defining an action of $\Stab(\sigma_{\mu}^0)$ on $\sigma_{\mu}^0$.
Note that $\GL_g(\Z)$ naturally acts on the set of quotients $\{\sigma_{\mu}^0/\Stab(\sigma_{\mu}^0)\}$;
we will denote by $\{[\sigma_{\mu}^0/\Stab(\sigma_{\mu}^0)]\}$ the (finite) orbits of this action.



\begin{thm}\label{AgS}
Let $\Sigma$ be a $\GL_g(\Z)$-admissible decomposition of $\Ort$.
The topological space $\Agt$ can be endowed
with the structure of a stacky fan with cells
$[\sigma_{\mu}^0/\Stab(\sigma_{\mu}^0)]$, which we denote by $\Agcom$.
\end{thm}
\begin{proof}
Fix a set $\S=\{\sigma_{\mu}^0/\Stab(\sigma_{\mu}^0)\}$ of representatives for
the orbits $[\sigma_{\mu}^0/\Stab(\sigma_{\mu}^0)]$.
For each element $\sigma_{\mu}^0/\Stab(\sigma_{\mu}^0)\in \S$, consider the
continuous map
$$\alpha_{\mu}:\frac{\sigma_{\mu}}{\Stab(\sigma_{\mu}^0)}\to\Agt,$$
induced by the inclusion $\sigma_{\mu} \hookrightarrow \Ort$.
By the definition of $\Agt$ it is clear that $\alpha_{\mu}$ sends
$\sigma_{\mu}^0/\Stab(\sigma_{\mu}^0)$ homeomorphically onto its image and also that
$$\bigcup \alpha_{\mu}\left(\frac{\sigma_{\mu}^0}{\Stab(\sigma_{\mu}^0)}\right)=\Agt,$$
where the union runs through all the elements of $\S$. Therefore the first
two conditions of definition \ref{pol-compl} are satisfied.
Let us check the condition \ref{pol-compl}(\ref{pol3}). Consider two elements
$\{\sigma_{\mu_1}^0/\Stab(\sigma_{\mu_1}^0)\}$ and
$\{\sigma_{\mu_2}^0/\Stab(\sigma_{\mu_2}^0)\}$ of $\S$. Clearly,
the intersection of the images of $\sigma_{\mu_1}/\Stab(\sigma_{\mu_1}^0)$
and $\sigma_{\mu_2}/\Stab(\sigma_{\mu_2}^0)$ in $\Agt$ can be written
in the form
$$\alpha_{\mu_1}\left(\frac{\sigma_{\mu_1}}{\Stab(\sigma_{\mu_1}^0)}\right)
\cap \alpha_{\mu_2}\left(\frac{\sigma_{\mu_2}}{\Stab(\sigma_{\mu_2}^0)}\right)
=\coprod_i \alpha_{\nu_i}\left(\frac{\sigma_{\nu_i}^0}{\Stab(\sigma_{\nu_i}^0)}\right),$$
where $\sigma_{\nu_i}^0/\Stab(\sigma_{\nu_i}^0)$ are the elements of $\S$ such that
there exist elements $h_{i1}, h_{i2}\in \GL_g(\Z)$ such that $h_{i1}\sigma_{\nu_i}
h_{i1}^t$ is a face of the cone $\sigma_{\mu_1}$ and $h_{i2}\sigma_{\nu_i}
h_{i2}^t$ is a face of the cone $\sigma_{\mu_2}$. Note that the above elements
$h_{i1}$ and $h_{i2}$ are not unique, but we will fix a choice for them
in what follows.
We have to find an integral linear map $L:\langle\sigma_{\mu_1}\rangle=\R^{m_{\mu_1}}
\to \langle \sigma_{\mu_{2}}\rangle=\R^{m_{\mu_2}}$ making the following diagram commutative
\begin{equation}\label{lin-diag2}
\xymatrix{
\coprod_i \alpha_{\nu_i}\left(\frac{\sigma_{\nu_i}^0}{\Stab(\sigma_{\nu_i}^0)}\right)
\ar@{^{(}->}[rd] \ar@{^{(}->}[r]&
\alpha_{\mu_1}\left(\frac{\sigma_{\mu_1}}{\Stab(\sigma_{\mu_1}^0)}\right)& \sigma_{\mu_1} \ar@{^{(}->}[r] \ar@{->>}[l]
\ar[d]^L& \langle \sigma_{\mu_1}\rangle=\R^{m_{\mu_1}} \ar[d]^{L} \\
&\alpha_{\mu_2}\left(\frac{\sigma_{\mu_2}}{\Stab(\sigma_{\mu_2}^0)}\right)& \sigma_{\mu_2}
 \ar@{^{(}->}[r] \ar@{->>}[l]& \langle \sigma_{\mu_2}\rangle=\R^{m_{\mu_2}} .}
\end{equation}
Consider the integral linear maps
$$\begin{sis}
\pi_i:&\langle \sigma_{\mu_1}\rangle=\R^{m_{\mu_1}}\stackrel{\w{\pi_i}}{\twoheadrightarrow}
\langle \rho(h_{i1})(\sigma_{\nu_i})\rangle \stackrel{\rho(h_{i1}^{-1})}{\longrightarrow}
\langle \sigma_{\nu_i}\rangle:=\R^{m_{\nu_i}}, \\
\gamma_i:& \langle \sigma_{\nu_i}\rangle=\R^{m_{\nu_i}} \stackrel{\rho(h_{i2})}{\longrightarrow}
\langle \rho(h_{i2})(\sigma_{\nu_i})\rangle\stackrel{\w{\gamma_i}}{\hookrightarrow}
\langle \sigma_{\mu_2}\rangle=\R^{m_{\mu_2}},\\
\end{sis}$$
where $\w{\pi_i}$ is the orthogonal projection of $\langle \sigma_{\mu_1}\rangle$ onto its
subspace $\langle \rho(h_{i1})(\sigma_{\nu_i})\rangle$ and $\w{\gamma_i}$ is the natural inclusion
of $\langle \rho(h_{i2})(\sigma_{\nu_i})\rangle$ onto $\langle \sigma_{\mu_2}\rangle$.
We define the following  integral linear map
$$L:\R^{m_{\mu_1}}\stackrel{\oplus_i \pi_i}{\longrightarrow} \bigoplus \R^{m_{\nu_i}} \stackrel{\oplus_i \gamma_i}{\longrightarrow} \R^{m_{\mu_2}}.$$
It is easy to see that $L$ is an integral linear map making
the above diagram (\ref{lin-diag2}) commutative, and this concludes the proof.
\end{proof}

\subsection{Voronoi decomposition: $\AgV$}

Some $\GL_g(\Z)$-admissible decompositions of $\Ort$ have been studied in
detail in the reduction theory of positive definite quadratic forms (see \cite[Chap. 8]{NamT}
and the references there), most notably:
\begin{enumerate}[(i)]
\item The perfect cone decomposition (also known as the first Voronoi
decomposition);
\item The central cone decomposition;
\item The Voronoi decomposition (also known as the second Voronoi
decomposition or the L-type decomposition).
\end{enumerate}
Each of them plays a significant (and different) role in the theory of the
toroidal compactifications of the moduli space of principally
polarized abelian varieties (see \cite{Igu}, \cite{alex1}, \cite{SB}).

\begin{example}

In Figure \ref{VorFig} we illustrate a section of the
$3$-dimensional cone $\Omega_2^{\rm rt}$, where we represent just some of the infinite Voronoi
cones (which for $g=2$ coincide with the perfect cones and with the
central cones). For $g=2$, there is only one
$\GL_g(\Z)$-equivalence class of maximal dimensional cones, namely
the principal cone $\prin^0$ (see section \ref{first-Vor}). Therefore, all the maximal cones in the picture
will be identified in the quotient $A_g^{tr,V}$.

\begin{figure}[h]
\begin{center}
\scalebox{.5}{\epsfig{file=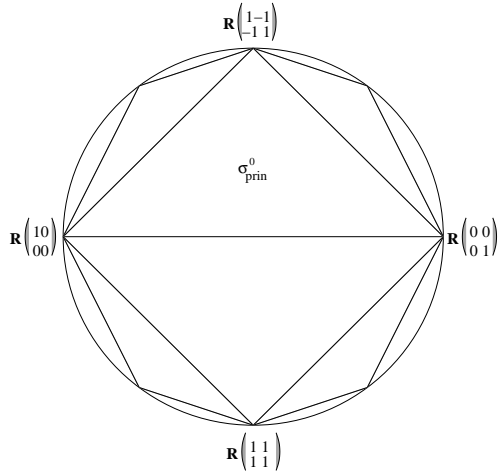}}
\end{center}
\caption{A section of $\Omega_2^{\rm rt}$ and its Voronoi
decomposition.}
\label{VorFig}
\end{figure}

\end{example}

Let us focus our attention on the Voronoi decomposition,
since it is the one that better fits in our setting.
It is based on the so-called {\it Dirichlet-Voronoi
polytope} $\Vor(Q)\subset \R^g$ associated to a positive semi-definite
quadratic form $Q\in \Ort$.
Recall (see for example \cite[Chap. 9]{NamT} or \cite[Chap. 3]{Val})
that if $Q\in \O$, then $\Vor(Q)$ is defined as
\begin{equation}\label{def-Vor}
\Vor(Q):=\{x\in \R^g \: :\: Q(x)\leq Q(v-x)
\text{ for all } v\in \Z^g\}.
\end{equation}
More generally, if $Q=h
\left(\begin{array}{cc}
Q' & 0 \\
0 &  0 \\
\end{array}\right)
h^t$ for some $h\in \GL_g(\Z)$ and some po\-si\-ti\-ve definite quadratic form $Q'$ in $\R^{g'}$, $0\leq g'\leq g$ (see Remark \ref{rat-qua}), then $\Vor(Q):=h^{-1}\Vor(Q')(h^{-1})^t\subset
h^{-1}\R^{g'}(h^{-1})^t$.
In particular, the smallest linear subspace containing $\Vor(Q)$
has dimension equal to the rank of $Q$.


\begin{defi}\label{Vor-dec}
The Voronoi decomposition $V=\{\sigma_{P}\}$ is the $\GL_g(\Z)$-admissible decomposition of
$\Ort$ whose open cones $\sigma_{P}^0:=\Int(\sigma_P)$ are parametrized by
Dirichlet-Voronoi polytopes $P\subset \R^g$ in the following way
$$\sigma_{P}^0:=\{Q\in \Ort\: : \: \Vor(Q)=P\}.$$
\end{defi}

\begin{remark}\label{parallelotopes}
The polytopes $P\subset \R^g$ that appear as Dirichlet-Voronoi polytopes of
quadratic forms in $\O$ are of a very special type: they are
{\it parallelohedra}, i.e. the set of translates of the form
$v+P$ for $v\in \Z^g$ form a face-to-face tiling of $\R^g$
(see for example \cite{McM2} or \cite[Chap. 3]{Val}).
Indeed, it has been conjectured by Voronoi (\cite{Vor})
that all the parallelohedra are affinely isomorphic to Dirichlet-Voronoi
polytopes (see \cite{DG} for
an account on the state of the conjecture).
\end{remark}

The natural action of $\GL_g(\Z)$ on the cones $\sigma_P^0$
corresponds to the natural action of $\GL_g(\Z)$ on the set of
all Dirichlet-Voronoi polytopes $P\subset \R^g$. We denote by $[P]$
(resp. $[\sigma^0_P]$) the equivalence class of $P$ (resp. $\sigma_P^0$)
under this action. We set also $C([P]):=[\sigma^0_P/\Stab(\sigma^0_P)]$.

\begin{defi}\label{AgV}
$\AgV$ is the stacky fan associated to the
Voronoi decomposition $V=\{\sigma_{P}\}$. Its cells are the $C([P])$'s as $[P]$ varies among the $\GL_g(\Z)$-equivalence classes
of Dirichlet-Voronoi polytopes in $\R^g$.
\end{defi}

In order to describe the maximal cells and codimension one cells of $\AgV$
(in analogy with Proposition \ref{prop-Mgt}), we need to introduce some
definitions. A Dirichlet-Voronoi polytope $P\subset \R^g$
is said to be \emph{primitive} if it is of dimension $g$
and the associated face-to-face tiling of $\R^g$ (see
Remark \ref{parallelotopes}) is such that at each
vertex of the tiling, the minimum number,
namely $g+1$, of translates of $P$ meet (see \cite[Sec. 2.2]{Val}).
A Dirichlet-Voronoi polytope $P\subset \R^g$
is said to be \emph{almost primitive} if it is of dimension $g$
and the associated face-to-face tiling of $\R^g$ (see
Remark \ref{parallelotopes}) is such that there is exactly one vertex,
modulo translations by $\Z^g$, where $g+2$ translates of $P$ meet
and at all the other vertices of the tiling only $g+1$ translates of
$P$ meet.

The properties of the following Proposition are the translation in our language of well-known
properties of the Voronoi decomposition (see the original
\cite{Vor} or \cite{Val} and the references there).
Unfortunately, the results we need are often stated in terms of the Delaunay decomposition,
which is the dual of the tiling of $\R^g$ by translates of the Dirichlet-Voronoi
polytope (see for example \cite[Chap. 9]{NamT} or \cite[Sec. 2.1]{Val}).
So, in our proof we will assume that the reader is familiar with the
Delaunay decomposition,
limiting ourselves to translate the above properties in terms of the Delaunay
decomposition and to explain how they follow from known results
about the Voronoi decomposition.

\begin{prop}\label{prop-AgV}
\noindent
\begin{enumerate}[(i)]
\item\label{AgV-pure}
The maximal cells of $\AgV$ are exactly those $C([P])$ such that
$P$ is primitive. $\AgV$ is of pure dimension $\binom{g+1}{2}$.
\item \label{AgV-conn} The codimension one cells of $\AgV$ are exactly
those of the form $C([P])$ such that $P$ is almost-primitive.
$\AgV$ is connected through codimension one.
\item \label{AgV-codim1}
Every codimension one cell of $\AgV$ lies in the closure
of one or two maximal cells.
\end{enumerate}
\end{prop}
\begin{proof}

The Dirichlet-Voronoi polytopes $P\subset \R^g$ that are primitive
correspond to Delaunay decompositions that are triangulations, i. e.
such that every Delaunay polytope is a simplex (see \cite[Sec. 3.2]{Val}).
The Dirichlet-Voronoi polytopes $P\subset \R^g$ that are almost-primitive
correspond to the Delaunay decompositions that  have
exactly one Delaunay repartitioning polytope, in the sense of
\cite[Sec. 2.4]{Val}, and all the other Delaunay polytopes are
simplices. Two maximal cells that have a common codimension one cell in their
closure are usually called bistellar neighbors (see \cite[Sec. 2.4]{Val}).
With this in mind, all the above properties follow from the (so-called)
\textit{Main Theorem of Voronoi's reduction theory} (see \cite{Vor} or
\cite[Thm. 2.5.1]{Val}).

\end{proof}


\subsection{Zonotopal Dirichlet-Voronoi polytopes: $\Agzon$}

Among all the Dirichlet-Voronoi polytopes, a remarkable subclass is represented by the
\emph{zonotopal} ones.
Recall (see \cite[Chap. 7]{Zie}) that a zonotope is a polytope that can be realized
as a Minkowski sum of segments, or equivalently, that can be obtained
as an affine projection of an hypercube.

\begin{remark}\label{Vorconj-zon}
Voronoi's conjecture has been proved for zonotopal parallelohedra
(see \cite{McM}, \cite{Erd}, \cite{DG2}, \cite{Val2}):
every zonotopal parallelohedron is affinely equivalent to
a zonotopal Dirichlet-Voronoi polytope.
Therefore, there is a bijection
$$
\left\{\begin{aligned}
& \text{Zonotopal parallelohedra}\\
& \text{in } \R^g \\ \end{aligned} \right\}_{/{\rm aff}}
\longleftrightarrow
\left\{\begin{aligned}
& \text{Zonotopal Dirichlet-Voronoi}\\
&\text{polytopes in }  \R^g \\
\end{aligned}\right\}_{/\GL_g(\Z)}
$$
\end{remark}

There is a close (and well-known) relation between zonotopal Dirichlet-Voronoi polytopes
$P\subset \R^g$ up to $\GL_g(\Z)$-action and regular matroids $M$
of rank at most $g$. We need to review this correspondence in
detail because it is crucial for the sequel of the paper and
also because we need to fix the notations we are going to use.
Consider first the following

\begin{constr}\label{constr-matrix}
Let $A\in M_{g,n}(\R)$ be a totally unimodular matrix of rank $r\leq g$.
Consider the linear map $f_{A^t}:\R^g\to \R^n$, $x\mapsto A^t\cdot x$,
where $A^t$ is the transpose of $A$. For any $n$-tuple
$\un{l}=(l_1,\ldots,l_n)\in \R_{>0}^n$, consider the
positive definite quadratic form $|| \cdot||_{\un{l}}$ on $\R^n$
given on $y=(y_1,\ldots,y_n)\in \R^n$ by
$$||y||_{\un{l}}:=l_1y_1^2+\ldots+l_ny_n^2,$$
and its pull-back $Q_{A,\un{l}}$ on $\R^g$ via $f_{A^t}$, i.e.
\begin{equation}\label{Q_A}
Q_{A,\un{l}}(x):=||A^t\cdot x||_{\un{l}},
\end{equation}
for $x\in\R^g$. Clearly $Q_{A,\un{l}}$ has rank equal to $r$ and belongs to
$\Ort$. As $\un{l}$ varies in $\R^n_{>0}$, the semi-positive
definite quadratic forms $Q_{A,\un{l}}$ form an open cone in $\Ort$
which we denote by $\sigma^0(A)$. Its closure in $\Ort$, denoted
by $\sigma(A)$, consists of the quadratic
forms $Q_{A, \un{l}}\in \Ort$, where $\un{l}$ varies in $\R_{\geq 0}^n$.
The faces of $\sigma(A)$ are easily seen to be of the form
$\sigma(A\setminus I)$ for some $I\subset \{1,\ldots,n\}$, where $A\setminus
I$ is the totally unimodular matrix obtained from $A$ by deleting the column vectors $v_i$ with $i\in I$.

Considering the column vectors $\left\{v_1,\ldots,v_n\right\}$ of $A$
as elements of $(\R^g)^*$, we define the following zonotope of $\R^g$:
\begin{equation}\label{Z_A}
Z_A:=\{x\in \R^g\, :\, -1/2\leq v_i(x)\leq 1/2 \text{ for } i=1,\cdots,n\}\subset \R^g.
\end{equation}
Its polar polytope (see \cite[Sec. 2.3]{Zie})
$Z_A^*\subset (\R^g)^*$ is given as a Minkowski sum:
\begin{equation}\label{dual-Z_A}
Z_A^*:=\left[-\frac{v_1}{2},+\frac{v_1}{2}\right]+\ldots+
\left[-\frac{v_n}{2},+\frac{v_n}{2}\right]\subset (\R^g)^*.
\end{equation}
Clearly the linear span of $Z_A$ has dimension $r$.

Finally, if $M$ is a regular matroid of rank $r(M)\leq g$,
write $M=M[A]$, where $A\in M_{g,n}(\R)$ is a totally unimodular matrix
of rank $r$ (see Theorem \ref{reg-thm}(i)). Note that if
$A=XBY$ for a matrix $X\in \GL_g(\Z)$ and a permutation
matrix $Y\in \GL_n(\Z)$, then $\sigma^0(A)=X\sigma^0(B)X^t$
and $Z_A= X \cdot Z_B$.
Therefore, according to Theorem \ref{reg-thm}(ii), the $\GL_g(\Z)$-equivalence
class of $\sigma^0(A)$, $\sigma(A)$ and of $Z_A$ depends only on the matroid $M$ and
therefore we will set $[\sigma^0(M)]:=[\sigma^0(A)]$, $[\sigma(M)]=[\sigma(A)]$
and $[Z_M]:=[Z_A]$. The matroid $M\setminus I=M[A\setminus I]$ for a subset
$I\subset E(M)=\{v_1,\cdots, v_n\}$ is called the {\it deletion} of $I$ from $M$
(see \cite[Pag. 22]{Oxl}).

\end{constr}

\begin{lemma}\label{cone-expl}
Let $A$ be as in \ref{constr-matrix}.
Then $Z_A$ is a Dirichlet-Voronoi polytope
whose associated cone is given by $\sigma^0(A)$, i.e. $\sigma_{Z_A}^0=\sigma^0(A)$.
\end{lemma}
\begin{proof}
Let us first show that
$\Vor(Q_{A,\un{l}})=Z_A$ for any $\un{l}\in \R_{>0}^n$, i.e. that
$Z_A$ is a Dirichlet-Voronoi polytope and that $\sigma^0(A)\subset \sigma_{Z_A}^0$.
Assume first that $A$ has maximal rank $r=g$ or, equivalently, that
$f_{A^t}:\R^g\to \R^n$ is injective.
By definitions (\ref{def-Vor}) and (\ref{Q_A}), we get that
\begin{equation*}\Vor(Q_{A,\un{l}})=\{x\in \R^g\,:\, ||f_{A^t}(x)||_{\un{l}}\leq
||f_{A^t}(\lambda - x)||_{\un{l}} \,\text{ for all } \lambda \in \Z^g\}.\tag{*}
\end{equation*}
The total unimodularity of $A$ and the injectivity of $f_{A^t}$
imply that the map $f_{A^t}:\R^g\to \R^n$
is integral and primitive, i.e. $f_{A^t}(x)\in \Z^n$ if and only if $x\in \Z^g$.
Therefore, from (*) we deduce that
\begin{equation*}
\Vor(Q_{A,\un{l}})=f_{A^t}^{-1}(\Vor(||\cdot ||_{\un{l}}).\tag{**}
\end{equation*}
Since $||\cdot||_{\un{l}}$ is a diagonal quadratic form on $\R^n$,
it is easily checked that
\begin{equation*}
\Vor(||\cdot||_{\un{l}})=\left[-\frac{e_1}{2},\frac{e_1}{2}\right]+\cdots
+ \left[-\frac{e_n}{2},\frac{e_n}{2}\right],\tag{***}
\end{equation*}
where $\{e_1,\cdots,e_n\}$ is the standard basis of $\R^n$. Combining
(**) and (***), and using the fact that $f_{A^t}(x)=(v_1(x),\cdots, v_n(x))$,
we conclude. The general case $r\leq g$ follows in a similar way
after replacing $\R^g$ with $\R^g/\Ker(f_{A^t})$. We leave the details
to the reader.

In order to conclude that $\sigma^0(A)=\sigma_{Z_A}^0$, it is enough
to show that the rays of $\sigma_{Z_A}$ are contained in $\sigma(A)$.
By translating the results of \cite[Sec. 3]{ER} into our notations,
we deduce that the rays of $\sigma_{Z_A}$ are all of the form
$\sigma_{Z(A)_i}$ for the indices $i$ such that $v_i\neq 0$, where
$$Z(A)_i:=Z(A)\bigcap_{j\neq i}\{v_j^*=0\}.
$$
By what we already proved, we have the inclusion $\sigma(v_i):=
\sigma(A\setminus \{i\}^c)\subset
\sigma_{Z(A)_i}$, where $\{i\}^c:=\{1,\cdots,n\}\setminus \{i\}$.
Since both the cones are one dimensional, we deduce that
$\sigma(A\setminus \{i\}^c)=\sigma_{Z(A)_i}$, which shows that
all the rays of $\sigma_{Z_A}$ are also rays of $\sigma(A)$.

\end{proof}

\begin{thm}\label{mat-zon}
\noindent
\begin{enumerate}[(i)]
 \item \label{mat-zon1}Given a regular matroid $M$ of rank $r(M)\leq g$, $[Z_M]$ is the
$\GL_g(\Z)$-equivalence class of a zonotopal Dirichlet-Voronoi
polytope and every such class arises in this way.
\item  \label{mat-zon2} If $M_1$ and $M_2$ are two regular matroids, then $[Z_{M_1}]=
[Z_{M_2}]$ if and only if $[\sigma(M_1)]=[\sigma(M_2)]$ if and only if
$\widetilde{M_1}=\widetilde{M_2}$.
\item \label{mat-zon3} If $M$ is simple, then any representative $\sigma(M)$ in $[\sigma(M)]$
is a simplicial cone of dimension $\#E(M)$ whose faces are of the form
$\sigma(M\setminus I)\in [\sigma(M\setminus I)]$ for some uniquely determined
$I\subset E(M)$.
\end{enumerate}
\end{thm}
\begin{proof}
The first assertion of (i) follows from the previous Lemma \ref{cone-expl}
together with the fact that each representative $Z_A\in [Z_M]$ is zonotopal
by definition (see \ref{constr-matrix}). The second assertion is a well-known
result of Shephard and McMullen (\cite{She}, \cite{McM} or
also \cite[Thm. 1]{DG2}).

Consider part (ii). By definition \ref{Vor-dec} and what remarked shortly after,
$[\sigma(M_1)]$$=[\sigma(M_2)]$ if and only if $[Z_{M_1}]=[Z_{M_2}]$.
Let us prove that $[Z_M]=[Z_{\w{M}}]$.
Write $M=M[A]$ as in \ref{constr-matrix}. From Definitions
\ref{rep-mat} and \ref{simplifi-mat}, it is straightforward to see
that $\w{M}=M[\w{A}]$, where $\w{A}$ is the totally unimodular
matrix obtained from $A$ by deleting the zero columns and, for each set
$S$ of proportional columns, deleting all but one distinguished column
of $S$. From the definition (\ref{Z_A}), it follows easily that
$Z_A=Z_{\w{A}}$, which proves that $[Z_M]=[Z_{\w{M}}]$.

To conclude part (ii),
it remains to prove that if $M_1$ and $M_2$ are simple regular matroids
such that $[Z_{M_1}]=[Z_{M_2}]$, then $M_1=M_2$.
We are going to use the poset of flats $\L(M)$ of a matroid $M$ (see
\cite[Sec. 1.7]{Oxl}). In the special case (which will be our case)
where $M=M[A]$ for some matrix
$A\in M_{g,n}(F)$ over some field $F$, whose column vectors
are denoted as usual by
$\{v_1,\cdots,v_n\}$, a flat (see \cite[Sec. 1.4]{Oxl})
is a subset $S\subset E(M)=\{1,\cdots,n\}$ such that
$${\rm span}(v_i\, : \, i\in S)\subsetneq {\rm span}(v_k, v_i\,:\, i\in S),
$$
for any $k\not\in S$. $\L(M)$ is the poset of flats endowed with the natural
inclusion. It turns out that (see \cite[Pag. 58]{Oxl}) for two matroids $M_1$ and $M_2$,
we have
\begin{equation*}
\L(M_1)\cong \L(M_2) \Leftrightarrow \w{M_1}=\w{M_2}. \tag{*}
\end{equation*}

Moreover, in the case where $M$ is a regular and simple matroid, $\L(M)$ is determined
by the $GL_g(\Z)$-equivalence class $[Z_M]$. Indeed, writing
$M=M[A]$ as in \ref{constr-matrix}, $Z_M$ determines, up to the natural action
of $GL_g(\Z)$, a central arrangement $\A_M$ of non-trivial and
pairwise distinct
hyperplanes in $(\R^g)^*$, namely those given by $H_i:=\{v_i=0\}$ for $i=1,\cdots, n$.
Denote by $\L(\A_M)$ the intersection poset of $\A_M$, i.e.
the poset of linear subspaces of $(\R^g)^*$ that are
intersections of some of the hyperplanes $H_i$, ordered by inclusion.
Clearly $\L(\A_M)$ depends only on the $GL_g(\Z)$-equivalence class $[Z_M]$.
It is easy to check that the map
\begin{equation*}
 \begin{aligned}
\L(M) & \longrightarrow \L(\A_M)^{{\rm opp}}\\
S & \mapsto \bigcap_{i\in S} H_i,
 \end{aligned}\tag{**}
\end{equation*}
is an isomorphism of posets, where $\L(\A_M)^{{\rm opp}}$
denotes the opposite poset of $\L(\A_M)$.
Now we can conclude the proof of part (ii). Indeed, if $M_1$ and $M_2$ are regular
and simple matroids such that $[Z_{M_1}]=[Z_{M_2}]$ then
$\L(\A_{M_1})\cong \L(\A_{M_2})$ which implies that $\L(M_1)\cong \L(M_2)$
by (**) and hence $M_1=M_2$ by (*).

Finally consider part (iii). Write $M=M[A]$ as in \ref{constr-matrix}
and consider the representative $\sigma(A)\in [\sigma(M)]$.
From \cite[Thm. 4.1]{ER}, we known that $\sigma(A)$ is simplicial.
We have already observed in \ref{constr-matrix} that all the faces of
$\sigma(A)$ are of the form
$\sigma(A\setminus I)$ for $I\subset E(M)=\{v_1,\cdots,v_n\}$
and that $\sigma(A\setminus I)\in [\sigma(M\setminus I)]$
by definition of deletion of $I$ from $M$.
In particular, the rays of $\sigma(A)$
are all of the form $\sigma(v_i):=\sigma(A\setminus \{v_i\}^c)$ for some
$v_i\in E(M)$, where $\{v_i\}^c:=E(M)\setminus \{v_i\}$.
The hypothesis that $M$ is simple (see \ref{simple-mat}) is equivalent to the fact that
the matrix $A$ has no zero columns and no parallel columns.
This implies that all the faces $\sigma(v_i)$ are $1$-dimensional and pairwise
distinct. Since $\sigma(A)$ is a simplicial cone,
its dimension is equal to the number of rays, i.e. to $n=\#E(M)$.
The fact that each face of $\sigma(A)$ is of the form $\sigma(A\setminus I)$
for a unique $I\subset E(M)$ follows from the fact  that in a simplicial
cone  each face is uniquely determined by the rays contained in it.

\end{proof}

From Theorem \ref{mat-zon}, it follows
that the class of all open Voronoi cones $\sigma_Z^0$
such that $Z\subset \R^g$ is a zonotopal Dirichlet-Voronoi polytope is stable
under the action of $\GL_g(\Z)$ and under the operation of taking faces
of the closures $\sigma_{Z}=\ov{\sigma_{Z}^0}$. Therefore the collection
of zonotopal Voronoi cones, i.e.
$$\Zon:=\{\sigma_{Z}\subset \Ort \: : \: Z\subset \R^g \text{ is zonotope}\},$$
is a $\GL_g(\Z)$-admissible decomposition of a closed subcone
of $\Ort$, i.e. $\Zon$ satisfies all the properties of Definition
\ref{decompo}  except the last one. Therefore we can give the following

\begin{defi}\label{Zon}
$\Agzon$ is the stacky subfan of $\AgV$ whose cells are of the form
$C([Z])$, where $[Z]$ varies among the $\GL_g(\Z)$-equivalence classes of
zonotopal Dirichlet-Voronoi polytopes in $\R^g$.
\end{defi}
$\Agzon$ has dimension $\binom{g+1}{2}$ but it is not pure-dimensional
if $g\geq 4$ (see Example \ref{small-dim} or \cite{DV} for the list of
maximal zonotopal cells for small values of $g$).
There is indeed only one zonotopal cell of maximal dimension
$\binom{g+1}{2}$, namely the one corresponding to the principal
cone (see section \ref{first-Vor} below).
Using the notations of \ref{constr-matrix}, given a regular
matroid $M$ of rank at most $g$, we set $C(M):=C([Z_M])$.
From Theorem \ref{mat-zon}, we deduce the following useful

\begin{cor}\label{faces-zon}
The cells of $\Agzon$ are of the form $C(M)$, where $M$ is a
simple regular matroid of rank at most $g$.
\end{cor}



We want to conclude this section on zonotopal Dirichlet-Voronoi polytopes
(and hence on zonotopal parallelohedra by remark \ref{Vorconj-zon})
by mentioning the following
\begin{remark}
Zonotopal parallelohedra $Z\subset \R^g$ are also closely related to
other geometric-combinatorial objects:
\begin{enumerate}[(i)]\label{Other-zon}
\item Lattice dicings of $\R^g$
(see \cite{ER});
\item Venkov arrangements of hyperplanes of $\R^g$
(see \cite{Erd});
\item Regular oriented matroids of rank at most $g$, up to reorientation
(see \\ \cite[Sec. 2.2, 6.9]{or-mat}).
\end{enumerate}
\end{remark}

\section{The tropical Torelli map}
\label{trop-tor}

\subsection{Construction of the tropical Torelli map $\tgt$}
\label{tgt}

We begin by defining the Jacobian of a tropical curve.

\begin{defi}\label{Jac}
Let $C=(\Gamma,w,l)$ be a tropical curve of genus $g$ and total weight
$|w|$. The {\it Jacobian} $\Jac(C)$ of $C$ is the tropical abelian variety
of dimension $g$ given by the real torus $(H_1(\Gamma,\R)\oplus \R^{|w|})/(H_1(\Gamma,\Z)\oplus\Z^{|w|})$ together with the semi-positive quadratic form $Q_{C}=Q_{(\Gamma,w,l)}$ which vanishes identically on $\R^{|w|}$ and is given on
$H_1(\Gamma,\R)$ as
\begin{equation}\label{def-quad}
Q_{C}\left(\sum_{e\in E(\Gamma)}\alpha_e \cdot e\right)=\sum_{e\in E(\Gamma)}
\alpha_e^2\cdot l(e).
\end{equation}
\end{defi}

\begin{remark}\label{rank=weight}
Note that the above definition is independent of the orientation
chosen to define $H_1(\Gamma,\Z)$. Moreover, after identifying
the lattice $H_1(\Gamma,\Z)\oplus\Z^{|w|}$ with $\Z^g$
(which amount to chose a basis of $H_1(\Gamma,\Z)$), we can
(and will) regard the arithmetic equivalence class of
$Q_{C}$ as an element of $\Ort$.
\end{remark}

\begin{remark}\label{comp-MK3}
The above definition of Jacobian is a generalization of the definition
of Mikhalkin-Zharkov (see \cite[Sec. 6]{MZ}).
More precisely, the Jacobian of a tropical curve of total weight zero
in our sense is the same as the Jacobian of Mikhalkin-Zharkov.
\end{remark}

\begin{example} \label{Ex-Peterson}
In Figure \ref{peterson} below,  the so-called Peterson graph is regarded as a tropical curve $C$ of genus $6$ with  identically zero weight  function
and with length function $l(e_i):=l_i\in \R_{> 0}, i=1,\dots, 15$.

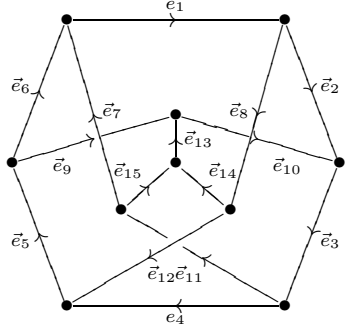
\begin{figure}[h]
\hspace{1cm}
\xymatrix@C=1.2pc@R=1pc{
&  *{\bullet} \ar@{-}[rrrr] |-{\SelectTips{cm}{}\object@{>}}^{\vec{e}_1} \ar@{-}[dddl] |-{\SelectTips{cm}{}\object@{<}}_{\vec{e}_6} \ar@{-}[ddddr] |-{\SelectTips{cm}{}\object@{<}}^{\vec{e}_7} &&&& *{\bullet} \ar@{-}[dddr] |-{\SelectTips{cm}{}\object@{>}}^{\vec{e}_2}\ar@{-}[ddddl]|!{[ddd];[lll]}\hole |-{\SelectTips{cm}{}\object@{>}}_{\vec{e}_8} & \\
&&&&&&\\
&  & & *{\bullet} \ar@{-}[d] |-{\SelectTips{cm}{}\object@{<}}^{\vec{e}_{13}} \ar@{-}[dlll]^>>>>>>{\vec{e}_9}|!{[d];[lll]}\hole |-{\SelectTips{cm}{}\object@{<}} \ar@{-}[drrr]|!{[d];[rrr]}\hole |-{\SelectTips{cm}{}\object@{>}}_>>>>>>{\vec{e}_{10}}&&&\\
*{\bullet} \ar@{-}[dddr] |-{\SelectTips{cm}{}\object@{<}}_{\vec{e}_5} &&  &*{\bullet} \ar@{-}[dl] |-{\SelectTips{cm}{}\object@{<}}_{\vec{e}_{15}} \ar@{-}[dr] |-{\SelectTips{cm}{}\object@{<}}^{\vec{e}_{14}}& && *{\bullet} \ar@{-}[dddl] |-{\SelectTips{cm}{}\object@{>}}^{\vec{e}_3} \\
&&*{\bullet} \ar@{-}[ddrrr]|!{[d];[rrr]}\hole |-{\SelectTips{cm}{}\object@{<}}_{\vec{e}_{11}}&&*{\bullet} \ar@{-}[ddlll] |-{\SelectTips{cm}{}\object@{>}}^{\vec{e}_{12}}&&\\
&&&&&&\\
& *{\bullet} \ar@{-}[rrrr] |-{\SelectTips{cm}{}\object@{<}}_{e_4} &&&&*{\bullet}
}
\caption{The Peterson graph $\Gamma$ endowed with an orientation.}\label{peterson}
\end{figure}

Fix an orientation of the edges as shown in the figure and 
consider the basis $B$ for the space $H_1(\Gamma,\R)=\R^6$ formed by the cycles $C_1,\dots,C_6$, where $C_1=\{\vec{e}_1,\vec{e}_2,\vec{e}_3,\vec{e}_4,\vec{e}_5,\vec{e}_6\}$, $C_2=\{\vec{e}_1,\vec{e}_2,\vec{e}_3,\vec{e}_{11},\vec{e}_7\}$,
$C_3=\{\vec{e}_1,\vec{e}_8,\vec{e}_{12},\vec{e}_5,\vec{e}_6\}$, $C_4=\{\vec{e}_3, $ $\vec{e}_{11},\vec{e}_{15},\vec{e}_{13},\vec{e}_{10}\}$,
$C_5=\{\vec{e}_5,\vec{e}_9,-\vec{e}_{13},-\vec{e}_{14},\vec{e}_{12}\}$ and $C_6=\{\vec{e}_1,\vec{e}_8,\vec{e}_{14},$ $-\vec{e}_{15},\vec{e}_7\}$.
Then the tropical Jacobian $J(C)$ of $C$ is the real torus $H_1(\Gamma,\R)/H_1(\Gamma,\Z)=\R^6/\Z^6$ endowed with the positive definite quadratic form $Q_C$ 
which is represented in the basis $B$ by the following matrix: 

\hspace{0.5cm}

\tiny{
%
$$\left(\begin{smallmatrix}
\sum_{i=1}^6 l_i & \frac{l_1+l_2+l_3}{2} & \frac{l_1+l_5+l_6}{2} & \frac{l_3}{2}  & \frac{l_5}{2} & \frac{l_1}{2} \\
&&&&&\\
\frac{l_1+l_2+l_3}{2} & l_1+l_2+l_3+l_{11}+l_7 & \frac{l_1}{2} & \frac{l_3+l_{11}}{2}  & 0 & \frac{l_1+l_7}{2}\\
&&&&&\\
\frac{l_1+l_5+l_6}{2} & \frac{l_1}{2} & l_1+l_5+l_6+l_8+l_{12} & 0  & \frac{l_5+l_{12}}{2} & \frac{l_1+l_8}{2}  \\
&&&&&\\
\frac{l_3}{2} & \frac{l_3+l_{11}}{2} & 0 & l_3+l_{10}+l_{11}+l_{13}+l_{15} &   \frac{-l_{13}}{2} & \frac{-l_{15}}{2}\\
&&&&&\\
\frac{l_5}{2} &0& \frac{l_5+l_{12}}{2} & \frac{-l_{13}}{2}  &    l_5+l_{9}+l_{12}+l_{13}+l_{14}&\frac{-l_{14}}{2}  \\
&&&&&\\
\frac{l_1}{2} & \frac{l_1+l_7}{2}& \frac{l_1+l_8}{2}&\frac{-l_{15}}{2}& \frac{-l_{14}}{2}&l_1+l_7+l_8+l_{14}+l_{15}
&&&&&\\
\end{smallmatrix}\right) $$
 }

\end{example}

Consider now the map (called tropical Torelli)
$$\begin{aligned}
\tgt: \Mgt & \to \AgV \\
C & \mapsto \Jac(C).
\end{aligned}$$

\begin{thm}\label{tr-map}
The above map $\tgt: \Mgt\to \AgV$ is a map of stacky fans.
\end{thm}
\begin{proof}
Let us first prove that $\tgt$ is a continuous map.
The map $\tgt$ restricted to the closure of one cell $\ov{C(\Gamma,w)}$ of
$\Mgt$ is clearly continuous since the quadratic form $Q_{C}$ on $H_1(\Gamma,\R)$
depends continuously on the lengths $l\in \R_{\geq 0}^{|E(\Gamma)|}$.
The continuity of $\tgt$ follows then from the fact that
$\Mgt$ is a quotient of $\coprod \ov{C(\Gamma,w)}$ with the induced
quotient topology.

Lemma \ref{tor-mat} below implies that $\tgt(C(\Gamma,w))\subset C\left(\w{M^*(\Gamma)}\right)$.
It remains to see that this map $\tgt:C(\Gamma,w) \to
C\left(\w{M^*(\Gamma)}\right)$ is
induced by an integral linear function $L_{(\Gamma,w)}$ between $\R^{|E(\Gamma)|}$
and the space $\R^{\binom{g(\Gamma)+1}{2}}$ of symmetric matrices on
$H_1(\Gamma,\R)$.
We define
\begin{equation}\label{lin-func}
\begin{aligned}
L_{(\Gamma,w)}: \R^{|E(\Gamma)|} & \longrightarrow \R^{\binom{g(\Gamma)+1}{2}},\\
l & \mapsto Q_{(\Gamma,w,l)},
\end{aligned}
\end{equation}
where $Q_{(\Gamma,w,l)}$ is defined by (\ref{def-quad}) above .
Clearly $L_{(C,\Gamma)}$ is an integral linear map that induces
the map $\tgt:C(\Gamma,\Z) \to C\left(\w{M^*(\Gamma)}\right)$.
This concludes the proof.
\end{proof}

\begin{lemma}\label{tor-mat}
The map $\tgt$ sends the cell $C(\Gamma,w)$ of $\Mgt$ surjectively onto
the cell $C\left(\w{M^*(\Gamma)}\right)$ of $\AgV$.
\end{lemma}
\begin{proof}
We use the construction in \ref{constr-matrix}.
Fixing an orientation of $\Gamma$, a basis of $H_1(\Gamma,\Z)$ and
an order of the edges of $\Gamma$, we get a natural inclusion
$$H_1(\Gamma,\Z)\cong  \Z^{g(\Gamma)}\hookrightarrow \Z^n\cong
C_1(\Gamma,\Z).
$$
The transpose of the integral matrix representing this
inclusion, call it $A^*(\Gamma)\in M_{g(\Gamma),n}(\Z)$, is well-known
to be totally unimodular and such that $M^*(\Gamma)=$ $M[A^*(\Gamma)]$
(see for example \cite[Ex. 6.4]{Zie}).

Now given a length function $l:E(\Gamma)\to \R_{>0}$, consider the
$n$-tuple $\un{l}\in \R_{>0}^n$ whose entries are the real positive
numbers $\{l(e)\}_{e\in E(\Gamma)}$ with respect to the order chosen on $E(\Gamma)$.
Comparing definitions (\ref{Q_A}) and (\ref{def-quad}), we deduce
that $Q_{A^*(\Gamma),\un{l}}=Q_{(\Gamma,w,l)}$.
The conclusion now follows from Lemma \ref{cone-expl} and Theorem \ref{mat-zon}.
\end{proof}

\subsection{Tropical Schottky}\label{Sch-section}

In this subsection, we want to prove a Schottky-type theorem, i.e.
we describe the image of the map $\tgt$.


We need to recall the following result (see \cite[3.1.1,
3.1.2, 3.2.1]{Oxl} for a proof).

\begin{lemma}
Let $\Gamma$ be a graph. For any subset
$I\subset E(\Gamma)=E(M^*(\Gamma))$, we have that
\begin{equation}\label{gr-contr}
M(\Gamma)\setminus I=M(\Gamma\setminus I)
\end{equation}
\begin{equation}\label{cogr-contr}
M^*(\Gamma)\setminus I=M^*(\Gamma/I)
\end{equation}
where $\Gamma\setminus I$ (resp. $\Gamma/I$)
is the graph obtained from $\Gamma$
by deleting (resp. contracting) the edges in $I$ and,
for a matroid $M$ and $I\subset E(M)$, we denote by $M\setminus I$
the matroid obtained from $M$ by deleting $I$.
\end{lemma}

From formula (\ref{cogr-contr}) and Theorem \ref{mat-zon}(\ref{mat-zon3}),
we deduce that the collection of cographic cones
$$\Cogr:=\{\sigma_Z\subset \Ort\: : \:  [\sigma_Z]=[\sigma(M)] \text{ for a cographic
matroid } M\}$$
is closed under taking faces of the cones, and therefore it defines a
$\GL_g(\Z)$-admissible decomposition of a closed subcone
of $\Ort$, i.e. $\Cogr$ satisfies all the properties of Definition
\ref{decompo}  except the last one.
Therefore we can give the following

\begin{defi}\label{Cogra}
$\Agco$ is the stacky subfan of $\Agzon\subset \AgV$ whose
cells are of the form $C(M)$, where $M$ is a simple cographic matroid
of rank at most $g$.
\end{defi}


The following Proposition summarizes some important properties of $\Agco$
(compare with Propositions \ref{prop-Mgt} and \ref{prop-AgV}).

\begin{prop}\label{max-Agco}
\noindent
\begin{enumerate}[(i)]
\item \label{Agco-cells}
The cells of $\Agco$ are of the form $C\left(M^*([\Gamma]_2)\right)$,
where $[\Gamma]_2$ varies among the 2-isomorphism classes of
$3$-edge-connected graphs of genus at most $g$.
\item \label{Agco-pure}
$\Agco$ has pure dimension $3g-3$ and its maximal cells are
of the form $C\left(M^*(\Gamma)\right)$, where $\Gamma$ is
$3$-regular and $3$-(edge)-connected.
\item \label{Agco-conn}
$\Agco$ is connected through codimension one.
\item\label{Agco-codim1}
All the codimension one cells of $\Agco$ lie in the closure of one,
two or three maximal cells of $\Agco$.
\end{enumerate}
\end{prop}
\begin{proof}
Part (\ref{Agco-cells}) follows by combining
Definition \ref{Cogra}, Remark \ref{2-isomo} and
Proposition \ref{simple-mat-graph}.

According to Theorem \ref{mat-zon}(\ref{mat-zon3}),
a cell $C(M^*([\Gamma]_2))$ of $\Agco$ is
of maximal dimension if and only if $\Gamma$ has the maximum number
of edges, and this happens precisely when $\Gamma$ is $3$-regular
in which case $\#E(\Gamma)=\dim C(M^*([\Gamma]_2))=3g-3$.
On the other hand, using the fact that every $3$-edge-connected
graph of genus $g$ is the specialization of a $3$-regular and $3$-edge-connected graph
(see \cite[Prop. A.2.4]{CV1}), formula
(\ref{cogr-contr}) and Theorem \ref{mat-zon}(\ref{mat-zon3}) give
that every cell of $\Agco$ is the face of some maximal dimensional
cell, i.e. $\Agco$ is of pure dimension $3g-3$.
To conclude the proof of part (\ref{Agco-pure}), it is enough to recall that
a $3$-edge-connected and $3$-regular graph $\Gamma$ is also $3$-connected
(see for example \cite[Lemma A.1.2]{CV1}) and that $[\Gamma]_2=\{\Gamma\}$
according to Fact \ref{3v}.

Using the same argument as in the beginning of the proof of
Proposition \ref{prop-Mgt}, it is easy to see that the codimension
one cells of $\Agco$ are of the form $C(M^*([\Gamma]_2))$, where
$[\Gamma]_2$ varies among the $2$-equivalence classes of genus $g$ graphs
having one vertex of valence $4$ and all the others of valence $3$
(it is easy to see that this property is preserved under
$2$-isomorphism). The same proof as in Proposition \ref{prop-Mgt}
gives now part (\ref{Agco-codim1}) while part (\ref{Agco-conn}) follows from
\cite[Thm. 3.3]{Cap}: any two $3$-regular and $3$-(edge)-connected graphs of the same genus are $3$-linked,
i.e. they can be obtained one from the other via a sequence of twisting operations as in
Figure \ref{twist-3reg} in such a way that each intermediate graph is also $3$-edge-connected.

\end{proof}


From  the above Proposition \ref{max-Agco} and Lemma \ref{tor-mat},
we deduce the following tropical Schottky theorem.

\begin{thm}\label{Sch}
The tropical Torelli map $\tgt$ is full and its image
is equal to the stacky subfan $\Agco\subset\AgV$.
\end{thm}

\begin{remark}
It is known (see Example \ref{small-dim} or \cite[Chap. 4]{Val}) that
$\Agco=\AgV$ if and only if $g\leq 3$. Therefore
$\tgt:\Mgt\to \AgV$ is surjective if and only if $g\leq 3$.
This has to be compared with the fact that the classical Torelli map
$t_g:M_g\to A_g$ is dominant if and only if $g\leq 3$.
\end{remark}

\subsection{Tropical Torelli}\label{tor-section}

In \cite[Thm. 4.1.9]{CV1}, the authors determine when two tropical
curves $C$ and $C'$ of total weight zero (i.e. tropical curves up to tropical
modifications in the sense of Mikhalkin-Zharkov)
are such that $\Jac(C)\cong \Jac(C')$. Indeed, we show here that
the same result extends easily to the more general case of tropical curves
(with possible non-zero weight).
We first need the following definitions.

\begin{defi}\label{2-isomo-trop}
Two tropical curves $C=(\Gamma, w, l)$ and $C'=(\Gamma',w',l')$
are $2$-isomorphic, and we write $C\equiv_2 C'$, if there exists a bijection
$\phi:E(\Gamma)\to E(\Gamma')$, commuting with the length functions $l$ and
$l'$, that induces a $2$-isomorphism between $\Gamma$ and $\Gamma'$.
We denote by $[C]_2$ the $2$-isomorphism equivalence class of a tropical
curve $C$.
\end{defi}

Similarly to definition \ref{3-conn}, we have the following

\begin{lemmadefi}\label{3-conn-trop}
Let $C=(\Gamma,l,w)$ a tropical curve. A $3$-edge-connecti\-vization
of $C$ is a tropical curve $C^3=(\Gamma^3,l^3,w^3)$ obtained
in the following manner:
\begin{enumerate}[(i)]
\item $\Gamma^3$ is a $3$-edge-connectivization of $\Gamma$
in the sense of definition \ref{3-conn}, i.e. $\Gamma^3$ is obtained
from $\Gamma$ by contracting all the separating edges of $\Gamma$
and, for each $C1$-set $S$ of $\Gamma$, all but one the edges of $S$,
which we denote by $e_S$;
\item $w^3$ is the weight function on $\Gamma^3$ induced by the weight
function $w$ on $\Gamma$ in the way explained in \ref{ps-trop}
viewing $\Gamma^3$ as a specialization of $\Gamma$;
\item $l^3$ is the length function on $\Gamma^3$ given
by $$l^3(e_S)=\sum_{e\in S} l(e),$$
for each C1-set $S$ of $\Gamma$.
\end{enumerate}
The $2$-isomorphism class of $C^3$ is well-defined;
it will be called the $3$-edge-connectivi\-zation class of $C$
and denoted by $[C^3]_2$.
\end{lemmadefi}

It is now easy to extend \cite[Thm. 4.1.9]{CV1} to the case of tropical
curves.

\begin{thm}\label{tor}
Let $C$ and $C'$ be two tropical curves of genus $g$.
Then $\tgt(C)=\tgt(C')$ if and only if $[C^3]_{2}=[C'^3]_{2}$.
In particular $\tgt$ is injective on the locus of $3$-connected
tropical curves.
\end{thm}
\begin{proof}
Note that $[C^3]_2=[C'^3]_2$ if and only if the $3$-edge-connectivizations
(in the sense of definition \cite[Def. 4.1.7]{CV1}) of the underlying metric
graphs $(\Gamma,l)$ and $(\Gamma',l')$ are cyclically equivalent
(in the sense of \cite[Def. 4.1.6]{CV1}), or in symbols
$[(\Gamma^3,l^3)]_{{\rm cyc}}=[(\Gamma'^3,l'^3)]_{{\rm cyc}}$.

On the other hand, from the definition \ref{Jac}, it follows that
$\Jac(C)\cong \Jac(C')$ if and only if the Albanese tori (in the sense
of definition \cite[4.1.4]{CV1}) of the underlying metric graphs
$(\Gamma,l)$ and $(\Gamma',l')$ are isomorphic, or in symbols
$\Alb(\Gamma,l)\cong \Alb(\Gamma',l')$.

With these two re-interpretations, the first assertion of the Theorem
follows from \cite[Thm 4.1.10]{CV1}. The second assertion follows from the
first and Fact \ref{3v}.
\end{proof}

Finally we can prove a tropical analogous of the classical Torelli
theorem which was conjectured by Mikhalkin-Zharkov in \cite[Sec. 6.4]{MZ}
and proved in \cite[Thm. A.2.1]{CV1} assuming the existence of the relevant
moduli spaces
(see \cite[Assumptions 1, 2, 3]{CV1}). However, since the conjectural
properties that these moduli spaces were assumed to have in
\cite{CV1} are slightly different from the
properties of the moduli spaces $\Mgt$ and $\AgV$ that we have
constructed here, we give a new proof of this result.

\begin{thm}\label{deg-one}
The tropical Torelli map $\tgt:\Mgt\to \AgV$ is of degree one
onto its image.
\end{thm}
\begin{proof}
The image of $\tgt$ is equal to $\Agco$ according to Theorem \ref{Sch}.
Therefore,
we have to prove that $\tgt:\Mgt\to \Agco$ satisfies the two conditions of
Definition \ref{maps}.

Proposition \ref{max-Agco} and Theorem \ref{tor} give that
a generic point of $\Agco$ is of the form $\Jac(C)$
for a unique tropical curve $C=(\Gamma,w,l)$,
whose underlying graph $\Gamma$ is $3$-regular and
$3$-connected. This proves that the first condition of
Definition \ref{maps} is satisfied.

It remains to prove that the integral linear function $L_{(\Gamma,w)}$, defined in (\ref{lin-func}),
is primitive for a tropical curve
$C=(\Gamma,w,l)$ whose underlying graph $\Gamma$ is
$3$-regular and $3$-connected.
So suppose that the quadratic form $Q_{(\Gamma,w,l)}$
on $H_1(\Gamma,\R)$ is integral, i.e. that the
associated symmetric bilinear form (which, by abuse
of notation, we denote by $Q_{(\Gamma,w,l)}(-,-)$)
takes integral values on $H_1(\Gamma,\Z)$; we have
to show that the length function $l$ takes integral values.
Since $\Gamma$ is $3$-edge-connected by hypothesis, every edge of $\Gamma$
is contained in a $C1$-set and all the $C1$-sets of $\Gamma$ have cardinality
one (see \ref{conn-girth}). Therefore, using \cite[Lemma 3.3.1]{CV1},
we get that for every edge $e\in E(\Gamma)$ there exist two cycles $\Delta_1$ and $\Delta_2$ of $\Gamma$ such that the intersection of their
supports is equal to $\{e\}$. By definition \ref{def-quad}, these two
cycles define two elements $C_1$ and $C_2$ of $H_1(\Gamma,\Z)$
(with respect to any chosen orientation of $\Gamma$) such that
$Q_{(\Gamma,w,l)}(C_1,C_2)=l(e)$. Since
$Q_{(\Gamma,w,l)}(-,-)$ takes integral values on $H_1(\Gamma,\Z)$
by hypothesis, we get that $l(e)\in \Z$, q.e.d.
\end{proof}

\section{Planar tropical curves and the principal cone}

\subsection{$\Aggr$ and the principal cone}\label{first-Vor}

Another important stacky subfan of
$\Agzon$ (other than $\Agco$)
is formed by the zonotopal cells that correspond
to graphic matroids. Indeed,
from formula (\ref{gr-contr}) and Theorem \ref{mat-zon}(\ref{mat-zon3}),
it follows that the collection of graphic cones
$$\Gr:=\{\sigma_Z\subset \Ort\: : \:  [\sigma_Z]=[\sigma(M)] \text{ for a graphic
matroid } M\}$$
is closed under taking faces of the cones, and therefore it defines a
$\GL_g(\Z)$-admissible decomposition of a closed subcone
of $\Ort$, i.e. $\Gr$ satisfies all the properties of Definition
\ref{decompo}  except the last one.
Therefore we can give the following

\begin{defi}\label{Gra}
$\Aggr$ is the stacky subfan of $\Agzon\subset \AgV$ whose
cells are of the form $C(M)$, where $M$ is a simple graphic matroid of rank at most $g$.
\end{defi}

By combining Corollary \ref{faces-zon}, Remark \ref{2-isomo} and
Proposition \ref{simple-mat-graph}, we get the following

\begin{remark}\label{con-gr}
The cells of $\Aggr$ are of the form $C(M([\Gamma]_2))$,
where $[\Gamma]_2$ varies among the 2-isomorphism classes of
simple graphs of cogenus at most $g$.
\end{remark}

$\Aggr$ is closely related to the so-called principal cone (Voronoi's principal domain
of the first kind), see \cite[Chap. 8.10]{NamT} and \cite[Chap. 2.3]{Val}.
It is defined as
$$\prin^0:=\{Q=(q_{ij})\in \O \: : \: q_{ij}<0 \text{ for } i\neq j, \:
\sum_{j} q_{ij}>0 \text{ for all } i.\}
$$
It is well-known that $\Stab(\prin^0)=S_{g+1}$ (see \cite[Sec. 2.3]{Val})
and we will denote by $\Cpr:=[\prin^0/\Stab(\prin^0)]$ the cell of
$\AgV$ corresponding to the principal cone $\prin^0$, and call it the
principal cell.

The following result is certainly well-known (see for example \cite[Sec. 3.5.2]{Val}),
but we include a proof here by lack of a proper reference.

\begin{lemma}\label{prin-K}
The $\GL_g(\Z)$-equivalence class $[\prin^0]$ of the principal
cone is equal to $[\sigma^0(M(K_{g+1}))]$, where $K_{g+1}$ is the complete
simple graph on $(g+1)$-vertices. Therefore $\Cpr=C(M(K_{g+1}))$ in $\AgV$.
\end{lemma}
\begin{proof}
Call $\{v_1,\cdots, v_{g+1}\}$ the vertices of $K_{g+1}$
and $e_{ij}$ (for $i<j$) the unique edge of
$K_{g+1}$ joining $v_i$ and $v_j$. Choose the orientation
of $K_{g+1}$ such that if $i<j$ then $s(e_{ij})=e_i$
and $t(e_{ij})=e_j$. It can be easily checked that
the elements $\{\delta(v_1),\cdots,\delta(v_g)\}$ form
a basis for $\Im(\delta)=H_1(K_{g+1},\Z)^{\perp}$.
Consider the transpose of the integral matrix, call it $A(K_{g+1})$,
that gives the inclusion $H_1(K_{g+1},\Z)^{\perp}\hookrightarrow
C_1(K_{g+1},\Z)$ with respect to the basis $\{\delta(v_1),\cdots,
\delta(v_g)\}$ and $\{e_{ij}\}_{i<j}$. In other words
\begin{equation*}
A(K_{g+1})^t\cdot \delta(v_k)=\sum_{i<k} e_{ik}-\sum_{k<j} e_{kj}.\tag{*}
\end{equation*}
Observe that $A(K_{g+1})\in M_{g,n}(\Z)$ where $n=\binom{g+1}{2}=
\#E(K_{g+1})$. It is well-known (see \cite[Prop. 5.1.2, 5.1.3]{Oxl}) that
$A(K_{g+1})$ is totally unimodular and that $M(K_{g+1})=M[A(K_{g+1})]$.

We now apply the construction in \ref{constr-matrix} to this matrix
$A(K_{g+1})$. For a $n$-tuple $\un{l}=(l_{ij})_{i<j}\in \R_{>0}^n$
(setting $l_{j,i}=l_{i,j}$ if $i<j$),
consider the quadratic form $Q_{A(K_{g+1}),\un{l}}$
of formula (\ref{Q_A}).
For the associated bilinear symmetric form, which we denote  $Q_{A(K_{g+1}),\un{l}}(-,-)$ (by an abuse of notation), we can compute,
using (*) above, that (for $i\neq j$)
\begin{equation*}
\begin{sis}
& Q_{A(K_{g+1}),\un{l}}(\delta(v_i),\delta(v_i))=\sum_{1\leq k\neq i\leq g}
l_{k,i} +l_{i,g+1}, \\
& Q_{A(K_{g+1}),\un{l}}(\delta(v_i),\delta(v_j))=-l_{i,j}.
\end{sis}
\end{equation*}
This easily implies that $\sigma^0(A(K_{g+1}))=\prin^0$, which concludes
the proof since, as observed before,
$[\sigma^0(A(K_{g+1}))]=[\sigma^0(M(K_{g+1}))]$.
\end{proof}

From the previous Lemma, we deduce the following

\begin{prop}\label{desc-gr}
The stacky subfan $\Aggr$ of $\Agzon \subset \AgV$
coincides with the closure inside $\Agzon$ (or $\AgV$) of the principal cell
$\Cpr$.
In particular it has pure dimension equal to $\binom{g+1}{2}$
and $\Cpr$ is the unique maximal cell.
\end{prop}
\begin{proof}
Consider the closure, call it $\ov{\Cpr}$, of $\Cpr$
inside $\AgV$. Note that $\Cpr\subset \Aggr$, because of the above
Lemma \ref{prin-K}, and therefore we get that $\ov{\Cpr}\subset
\Aggr$. In order to prove equality, consider a cell of $\Aggr$,
which, according to Remark \ref{con-gr}, is of the form
$C(M([\Gamma]_2))$, for a simple graph $\Gamma$ of cogenus at
most $g$. Such a graph can be obtained by $K_{g+1}$ by
deleting some edges and therefore, using  Theorem
\ref{mat-zon}(\ref{mat-zon3}) and formula (\ref{gr-contr}), we get
that $C(M([\Gamma]_2))$ is a face of the closure
of $C(M(K_{g+1}))=\Cpr$, and hence it belongs to
$\ov{\Cpr}$, q.e.d.
\end{proof}


\begin{remark}
The principal cone $\prin^0$ has many important properties, among which
we want to mention the following
\begin{enumerate}[(i)]
\item $\Cpr$ is the unique zonotopal cell
of maximal dimension $\binom{g+1}{2}$
(see \cite[Sec. 3.5.3]{Val} and the references there);
\item The Dirichlet-Voronoi polytope associated to $[\prin^0]$ is the permutahedron of dimension $g$ (see \cite[Ex. 0.10]{Zie}),
which is an extremal Dirichlet-Voronoi polytope in the sense that it has
the maximum possible number of $d$-dimensio\-nal faces among all Dirichlet-Voronoi polytopes of dimension $g$
(see \cite[Sec. 3.3.2]{Val} and the references there);
\item $\prin^0$ is the unique Voronoi cone that is also a
perfect cone  (see \cite{Dic}).
\end{enumerate}
\end{remark}

\subsection{Tropical Torelli map for planar tropical curves}

We begin with the following

\begin{defi}\label{planar}
We say that a tropical curve $C=(\Gamma, w, l)$
(resp. a stable marked graph
$(\Gamma,w)$) is {\it planar} if the underlying graph $\Gamma$
is planar.
\end{defi}

Note that the specialization of a planar tropical curve
is again planar. Therefore it makes sense to give the following
\begin{defi}\label{Mg-planar}
$\Mgpl$ is the stacky subfan of $\Mgt$
consisting of planar tropical curves.
\end{defi}

It is straightforward to check that any planar tropical
curve can be obtained as a
specialization of a 3-regular planar tropical curve.
Therefore we get the following

\begin{remark}
$\Mgpl$ is of pure dimension $3g-3$
with cells $C(\Gamma,w)\subset \R^{|w|}$, for planar stable marked
graphs $(\Gamma,w)$ of genus $g$. A cell $C(\Gamma,w)$ of $\Mgpl$ is
maximal if and only if $\Gamma$ is $3$-regular.
\end{remark}


We want now to describe the image of $\Mgpl$ under the map
$\tgt$. With that in mind, we consider the locus inside
$\Agzon$ formed by the zonotopal cells corresponding to matroids
that are at the same time graphic and cographic. Indeed,
from formulas (\ref{gr-contr}), (\ref{cogr-contr}) and Theorem \ref{mat-zon}(\ref{mat-zon3}),
it follows that the collection of cones
$$\Grcogr:=\{\sigma_Z\: : \:  [\sigma_Z]=[\sigma(M)] \text{ for a
graphic and cographic matroid } M\}$$
is a $\GL_g(\Z)$-admissible decomposition of a closed subcone
of $\Ort$, i.e. $\Grcogr$ satisfies all the properties of Definition
\ref{decompo}  except the last one. Therefore we can give the following

\begin{defi}\label{Gr-cogr}
$\Agcogr$ is the stacky subfan  of
$\Agzon\subset \AgV$ whose cells are of the form $C(M)$, where $M$ is a simple graphic and cographic matroid of rank at most $g$.
\end{defi}

Equivalently, $\Agcogr$ is the intersection of $\Agco$ and $\Aggr$
inside $\Agzon$.
Using Corollary \ref{faces-zon}, Proposition \ref{2dual},
Remark \ref{gr-cogr} and
Proposition \ref{simple-mat-graph}, we get the following

\begin{remark}\label{cells-grcogr}
The cells of $\Agcogr$ are of the form
$$C(M([\Gamma]_2))=C(M^*([\Gamma]_2^*)),$$
for $[\Gamma]_2$ planar and simple and $[\Gamma]_2^*$
the dual $2$-isomorphism class as in (\ref{dual-iso})
(which is therefore planar and $3$-edge-connected by (\ref{genus-dual})).
\end{remark}

\begin{example}\label{small-dim}
We have defined several stacky subfans of  $\AgV$,
namely:
$$\Agcogr\subset \Agco, \Aggr\subset \Agzon\subset \AgV.$$
For $g=2, 3$, they are all equal and they have a unique maximal
cell, namely the principal cell $\Cpr$ associated to the
principal cone $\prin^0$
(see \cite[Chap. 4.2, 4.3]{Val}).
However, for $g\geq 4$, all the above subfans are
different. For example, for $g=4$, we have that
(see \cite[Chap. 4.4]{Val}):
\begin{enumerate}[(i)]
\item $A_4^{{\rm tr, V}}$ has $3$ maximal cells
(of dimension $10$), one of which is $\Cpr$;
\item $A_4^{\rm zon}$ has two maximal cells: $\Cpr$ of
dimension $10$ and $C(M^*([K_{3,3}]_2))$ of dimension
$9$, where $K_{3,3}$ is the complete bipartite graph on
$(3,3)$-vertices;
\item $A_4^{\rm cogr}$ has two maximal cells (of dimension $9$):
$C(M^*([K_{3,3}]_2))$ and \\ $C(M^*([K_5-1]_2^*))$, where $K_5-1$ is the (planar) graph
obtained by the complete simple graph $K_5$ on $5$ vertices by
deleting one of its edges;
\item $A_4^{\rm gr}$ has a unique maximal cell (of dimension $10$),
namely $\Cpr$;
\item $A_4^{\rm gr, cogr}$ has a unique maximal cell (of dimension
$9$): $C(M^*([K_5-1]_2^*))= C(M([K_5-1]_2))$.
\end{enumerate}
Finally, we point out that $\Agzon$ becomes quickly much smaller than
$\AgV$ as $g$ grows: $A_5^{\rm tr, V}$ has 222 maximal cells while
$A_5^{\rm zon}$ only $4$; $A_6^{\rm tr, V}$ has more than $250,000$
maximal cells (although the exact number is still not known) while
$A_6^{\rm zon}$ only $11$ (see \cite[Chap. 4.5, 4.6]{Val} and
\cite[Sec. 9]{DV}).

\end{example}

Now, we can prove the main result of this section.

\begin{thm}\label{tor-planar}
The following diagram
$$\xymatrix{
\Mgpl\ar@{^{(}->}[r] \ar^{\tgt}[d] & \Mgt \ar^{\tgt}[d] \\
\Agcogr \ar@{^{(}->}[r] & \Agco.
}$$
is cartesian.
In particular, the map $\tgt: \Mgpl \to \Agcogr$ is full and of degree one.
\end{thm}
\begin{proof}
The fact that the diagram is cartesian follows from Lemma \ref{tor-mat}
together with the fact that $M^*(\Gamma)$ is graphic if and only
if $\Gamma$ is planar (see \ref{gr-cogr}).
The last assertion follows from the first and the Theorems \ref{Sch},
\ref{deg-one}.
\end{proof}


\subsection{Relation with the compactified Torelli map:
Namikawa's conjecture}\label{Nam-section}

In this last subsection, we use the previous results
to give a positive answer to a problem
posed by Namikawa (\cite[Problem (9.31)(i)]{NamT})
concerning the compactified (classical) Torelli map.

We need to recall first some facts about the classical Torelli
map and its compactification.
Denote by $\Mg$ the coarse moduli space of smooth and projective curves
of genus $g$, by $\Ag$ the coarse moduli space of principally polarized
abelian varieties of dimension $g$. The classical Torelli map
$$ \tg : \Mg\to \Ag,  $$
sends a curve $X$ into its polarized Jacobian $(\Jac(X), \Theta_X)$.

It was known to Mumford and Namikawa (see \cite[Sec. 18]{nam2},
or also \cite[Thm. 4.1]{alex}) that the Torelli
map extends to a regular map (called the compactified Torelli map)
\begin{equation}\label{comp-Torelli}
\tgb: \Mgb\to \Agb
\end{equation}
from the Deligne-Mumford moduli space $\Mgb$ of stable curves of
genus $g$  (see \cite{DM}) to the toroidal compactification $\Agb$
of $\Ag$ associated to the (second)
Voronoi decomposition (see \cite{AMRT}, \cite{NamT} or
\cite[Chap. IV]{FC}). The above map $\tgb$ admits also a modular
interpretation (see \cite{alex}), which was used in \cite{CV2}
to give a description of its fibers.

The moduli space $\Mgb$ admits a stratification into locally closed
subsets parame\-trized by stable weighted graphs $(\Gamma,w)$ of genus
$g$ (see definition \ref{mark-graph}). Namely, for each stable
weighted graph $(\Gamma,w)$ we can consider the locally closed
subset $S_{(\Gamma,w)}\subset \Mgb$ formed by stable curves of
genus $g$ whose weighted dual graph is isomorphic to $(\Gamma,w)$.
Observe that, given a stable curve $X$ with weighted dual graph
$(\Gamma,w)$, any smoothing of $X$ at a subset $S$
of nodes of $X$
has weighted dual graph equal to the specialization
of $(\Gamma,w)$ obtained by contracting the edges corresponding to
the nodes of $S$ (see \ref{ps-trop}).
From this remark, we deduce that:
\begin{equation}\label{strat-Mg}
C(\Gamma,w)\subset \ov{C(\Gamma',w')} \Leftrightarrow
\ov{S_{(\Gamma,w)}} \supset S_{(\Gamma',w')}.
\end{equation}

Similarly, from the general theory of toroidal compactifications
of bounded symmetric domains (see \cite{AMRT} or \cite{NamT}),
it follows that $\Agb$ admits a
stratification into locally closed subsets $S_{C([P])}$,
parametrized by the cells $C([P])$ of $\AgV$.
We have also that
\begin{equation}\label{strat-Ag}
 C([P])\subset C([P']) \Leftrightarrow
\ov{S_{C([P])}} \supset S_{C([P'])}.
\end{equation}

The compactified Torelli map respects the
toroidal structures of $\Mgb$ and $\Agb$ (see \cite[Thm. 4.1]{alex});
more precisely, we have that (compare with Lemma \ref{tor-mat}):
\begin{equation}\label{strata-tor}
\tgb(S_{(\Gamma,w)})\subset S_{C(\w{M^*(\Gamma)})}.
\end{equation}

Given a stacky subfan  $N$ of $\Mgt$
(in the sense of definition \ref{pol-compl}),
consider the union of all the strata $S_{(\Gamma,w)}$ of $\Mgb$ such that
$C(\Gamma,w)\in N$, and call it $U_N$. Similarly for any
stacky subfan of $\AgV$.  It is easily checked, using  formulas (\ref{strat-Mg}) and
(\ref{strat-Ag}), that such a $U_N$ is an open subset
of $\Mgb$ (resp. $\Agb$) containing $\Mg$ (resp. $\Ag$),
and thus it is a partial compactification
of $\Mg$ (resp. $\Ag$).

In particular we define $\Mgpla\subset \Mgb$ as the open subset
corresponding to the stacky subfan
$\Mgpl\subset \Mgt$ and $\Aggrco\subset \Agcogra\subset \Agb$
as the two open subsets corresponding to the two stacky subfans
 $\Agcogr\subset \Agco\subset \AgV$.

Observe that from formula (\ref{strata-tor}) it follows that
the compactified Torelli map $\tgb$ takes values in $\Agcogra$.
Finally we can state the main result of this subsection.

\begin{cor}\label{Nam-conj}
Given a stable curve $X$, we have that $\tgb(X)\in \Aggrco$
if and only if the dual graph $\Gamma_X$ of $X$ is planar.
\end{cor}
\begin{proof}
From formula (\ref{strata-tor}), it follows that $\tgt(X)\in
S_{C(\w{M^*(\Gamma_X)})}$. Therefore $\tgt(X)$
$\in \Aggrco$
if and only if $\w{M^*(\Gamma_X)}$ is a graphic matroid.
By the definition \ref{simplifi-mat} of the simplification of a matroid,
it follows easily that $\w{M^*(\Gamma_X)}$ is a graphic matroid
if and only if $M^*(\Gamma_X)$ is a graphic matroid.
By combining Proposition \ref{2dual} and Theorem
\ref{exi-dual}, we finally get that $M^*(\Gamma_X)$ is a graphic matroid
if and only if $\Gamma_X$ is planar.
\end{proof}

The part if of the above Corollary was proved (using analytic
techniques) by Namikawa in \cite[Thm. 5]{nam1}. The converse
was posed as a problem in \cite[Problem (9.31)(i)]{NamT}.

\section{Open questions and future plans}

In this section, we want to mention some of the many questions that arise in connection to our
work and on which we hope to come back in a near future:

\begin{enumerate}

\item  It is possible to make sense of the balancing condition for a stacky fan?
In other words, is it possible to define a tropical stacky fan? In particular,
can our moduli spaces $\Mgt$ and $\AgV$ be endowed with the structure of tropical
stacky fans?

\item What can be said about the topology of the tropical moduli spaces
$\Mgt$ and $\AgV$ that we have constructed? Can the study of these tropical topological
spaces share some light on the topology of the classical moduli spaces $\Mg$ and
$\Ag$ and of their compactifications $\Mgb$ and $\Agb$?

\item Generalize the construction of the moduli space $\Mgt$ to the
construction of the moduli space $M_{g,n}^{\rm tr}$ of $n$-pointed
tropical curves of genus $g$. Even more generally, construct the moduli space
of tropical maps $M_{g,n}^{\rm tr}(\R^N, \Delta)$ from $n$-pointed
tropical curves of genus $g$ to $\R^N$ with degree $\Delta$.
(The genus $g=0$ case is solved in \cite{GKM}).

\item  Recently, Lucia Caporaso has constructed a modular compactification $\ov{\Mgt}$ of $\Mgt$ (see \cite[Sec. 6]{Cap}). 
Construct a modular compactification $\ov{\AgV}$ of $\AgV$ and extend the tropical
Torelli map $\tgt$ to a map $\ov{\tgt}:\ov{\Mgt}\to \ov{\AgV}$.

\item (The first half of this problem was suggested to us by Bernd Sturmfels).

Recall that in classical geometry, the class in integral cohomology of
a smooth and projective curve $C$ of genus $g\geq 2$ embedded in its Jacobian
$\Jac(C)$ via an Abel-Jacobi map ${\rm alb}:C \hookrightarrow \Jac(C)$
is equal to
$$[C]=\frac{[\Theta_C]^{g-1}}{(g-1)!}\in H^{g-1}(\Jac(C),\Z) \hspace{1cm} \text{(Poincar\'e
formula)},$$
where $\Theta_C$ is the principal polarization induced by the canonical
theta divisor. Moreover, starting with an arbitrary principally polarized abelian variety
$(A, \Theta)$ of dimension $g\geq 2$,
the Matsusaka-Ran criterion says that the integral cohomological
class $\frac{[\Theta]^{g-1}}{(g-1)!}$ is represented by an effective
$1$-cycle $[D]$ if and only if $(A, \Theta)\cong (\Jac(C), \Theta_C)$ for
a smooth and projective curve $C$ of genus $g$ and $D=C$ embedded via some
Abel-Jacobi map (This criterion is in particular a geometric solution of
the Schottky problem).

In tropical geometry, the intersection in the cohomology ring of a variety
can be replaced by the stable intersection (see \cite{Stur}). Therefore the
following two questions seem very natural:
\begin{enumerate}
\item Is there a tropical Poincar\'e formula? Can such a formula help
to recover geometrically the tropical Torelli Theorem \ref{deg-one}?
\item Is there a tropical Matsusaka-Ran criterion? Can such a criterion
provide a geometric solution to the Schottky problem, complementary
to the combinatorial solution proposed in Theorem \ref{Sch}?
\end{enumerate}

\item Recall that in classical algebraic geometry, a well-known finite cover of $\Mg$ is
the moduli space $\Rg$ parametrizing non-trivial double \'etale covers $\w{C}\to C$ such that
$C$ is a smooth and projective curve of genus $g\geq 2$. There is a map
${\Pr}_g:\Rg\to \mathcal{A}_{g-1}$, called the Prym map,
sending a  double \'etale cover $\w{C}\to C$
into its Prym variety ${\rm Prym}(\w{C}\to C)$. A study of a
compactified rational map (and in particular a study of its indeterminacy locus)
between a modular compactification $\ov{\Rg}$ of $\Rg$
and $\ov{\mathcal{A}_{g-1}}^V$ has been carried out in \cite{ABH}.

We ask for tropical analogous of these classical results:
\begin{enumerate}
\item Construct a space $R_g^{\rm tr}$ parametrizing double \'etale covers
$\w{C}\to C$ between tropical curves such that $C$ has genus $g\geq 2$.
\item Define a tropical Prym map ${\rm Pr}_g^{\rm tr}:R_g^{\rm tr}\to A_{g-1}^{\rm tr, V}$
and study the fibers and the image of ${\rm Pr}_g^{\rm tr}$.
\end{enumerate}

\item Another well-known finite cover of $\Mg$ is
the moduli space $\Sg$ parametrizing spin curves of genus $g$, i.e. pairs
$(C, \eta)$ such that $C$ is a smooth and projective curve of genus $g\geq 2$
and $\eta$ is a theta-characteristic. There is a natural spin-Torelli
${\rm st}_g:\Sg\to \mathcal{N}_{g}$, where $\mathcal{N}_g$ is the moduli
space of dimension $g$ abelian varieties
with a principal theta-level structure, i.e. pairs $(A, \Theta)$ such that $A$
is an abelian variety of dimension $g$ and $\Theta$ is an effective
symmetric divisor defining a principal polarization on $A$.
A study of the corresponding
compactified rational map (and in particular a study of its indeterminacy locus)
between a modular compactification $\ov{\Sg}$ of $\Sg$
and a modular compactification $\ov{\mathcal{N}_g}^V$ of $\Ng$
is obtained in \cite{MV}.

We ask for tropical analogous of these classical results:
\begin{enumerate}
\item Construct a space $S_g^{\rm tr}$ parametrizing tropical spin
curves of genus $g$ and a space $N_g^{\rm tr, V}$ parametrizing
tropical abelian varieties with principal theta-level structure.
\item Define a tropical spin-Torelli map ${\rm st}_g^{\rm tr}:
s_g^{\rm tr}\to N_g^{\rm tr, V}$ and study the fibers and the image
of ${\rm st}_g^{\rm tr}$.
\end{enumerate}

\end{enumerate}

\end{document}